\documentclass[oneside]{amsart}
\usepackage{geometry}

\usepackage{amsmath,amsfonts,amssymb,amsthm}
\usepackage{mathtools}
\usepackage{enumitem}
\usepackage{tikz-cd}
\usepackage{color}
\usepackage{comment}
\usepackage{tikz}
\usepackage{lipsum}
\usepackage{tikz-3dplot}
\usetikzlibrary{decorations.pathmorphing}
\usepackage{thm-restate}
\usepackage{wrapfig}
\usepackage{float}
\usepackage[textsize=footnotesize]{todonotes} 
\usepackage{cite}

\usetikzlibrary{shapes.geometric, positioning, patterns}
\usetikzlibrary{calc}
\usepackage[dvipsnames]{xcolor}

\DeclareMathOperator{\col}{Col}

\newcommand{\str}[1]{\mathbb{#1}}
\newcommand{\cat}[1]{\mathcal{#1}}

\newcommand{\pol}{\operatorname{Pol}}

\newcommand{\CSP}{\operatorname{CSP}}
\newcommand{\Pol}{\operatorname{Pol}}

\newcommand{\idem}{\mathrm{idem}}

\declaretheorem[
	name=Theorem,
	numberwithin=section
	]{theorem}
\declaretheorem[
	name=Lemma,
	sibling=theorem,
	]{lemma}
\declaretheorem[
	name=Proposition,
	sibling=theorem,
	]{proposition}

\declaretheorem[
	name=Definition,
	style=definition,
    sibling=theorem
	]{definition}

\declaretheorem[
	name=Remark,
	style=remark,
	numbered=no
	]{remark}	
\declaretheorem[
	name=Example,
	style=remark,
	numbered=no
	]{exam}

\declaretheorem[name=Claim,
    style=remark,
    numbered=no
    ]{claim}

\newcommand{\op}{^{\mathrm{op}}}
\newcommand{\set}{\mathrm{set}}

\DeclareMathOperator{\id}{id}

\DeclareFontFamily{U}{hira}{}
\DeclareFontShape{U}{hira}{m}{n}{<-> udmj30}{}

\title[Cut-homotopies and the complexity of edge-coloring problems]{Cut-homotopies and the complexity\\ of edge-coloring problems} 

\author{Alexey Barsukov}

\author{Roman Feller}
\address{Technische Universit\"{a}t Wien, Austria}
\email{roman.feller@tuwien.ac.at}

\author{Maximilian Hadek}
\address{Faculty of Mathematics and Physics, Charles University, Czechia}
\email{maximilian.hadek@matfyz.cuni.cz}

\author{Davide Perinti}
\address{Technische Universit\"{a}t Hamburg, Germany}
\email{davide.perinti@tuhh.de}

\thanks{\noindent The first three authors are supported by the European Unions ERC Synergy Grant 101071674, POCOCOP. Views and opinions expressed are however those of the authors only and do not necessarily reflect those of the European Union or the European Research Council Executive Agency. Neither the European Union nor the granting authority can be held responsible for them. Davide Perinti is supported by the Deutsche Forschungsgemeinschaft (DFG) under the project number 534904934.
}

\usepackage[colorlinks = true,
            linkcolor = blue,
            urlcolor  = blue,
            citecolor = blue,
            anchorcolor = blue]{hyperref}

\begin{document}

\begin{abstract}
    We study the computational complexity of problems that ask if a given graph admits an edge-coloring that does not contain an edge-colored clique from some fixed finite family.
    We show that every such problem is poly-time equivalent to a Constraint Satisfaction Problem, yielding a P vs.\ NP-complete dichotomy.
    Our main contribution lies in the reduction from the CSP to the coloring problem where we apply methods from Ramsey theory and a novel notion of cut-homotopy.
\end{abstract}

\makeatletter
\let\oldmakefntext\@makefntext
\renewcommand{\@makefntext}[1]{\noindent#1}
\maketitle
\let\@makefntext\oldmakefntext
\makeatother

\thispagestyle{empty}

\vspace{-0.4cm}

\section{Introduction}\label{section:introduction}

Fix a finite family $\cat F$ of edge-colored graphs and let $A$ be the set of occurring colors. Let $\col(\cat F)$ be the following computational problem:
\begin{center}\vspace{0.75em}
   {\it \parbox{0.85\textwidth}{Is there an edge-coloring of a given input graph using colors from $A$ so that no member of $\cat F$ appears as an induced edge-colored subgraph?} } \vspace{0.75em}
\end{center}
If there is such an edge-coloring we refer to it as an \emph{$\cat F$-free coloring}.
Clearly, every such problem $\col( \cat F)$ is contained in  NP; moreover, a recent result due to Kun and Ne\v{s}et\v{r}il~\cite[Th.~4.1]{KunNesetril} asserts that the class of computational problems arising in this fashion is NP-rich. 
In search of subclasses avoiding NP-richness a natural next step is therefore to restrict the shape of forbidden graphs.
In the present paper, we consider the case where every member of $\cat F$ is a colored clique.
One of the earliest works in this direction dates back to Garey and Johnson's list of NP-complete problems~\cite{GareyJohnson} which contains the problem $\col(\cat F)$, where $\cat F$ consists of monochromatic triangles in two colors.
Later, this was extended to monochromatic cliques of arbitrary size $\geq 3$~\cite{Burr} and to arbitrarily many colors~\cite{SiggersCliques, BarsukovMottetPerinti}.
We extend these hardness results to a P vs.\ NP-complete dichotomy for $\col(\cat F)$, when $\cat F$ consists of arbitrary edge-colored cliques.
\begin{theorem}\label{thm:main}
    For every finite family $\mathcal F$ of edge-colored cliques, the problem $\col(\mathcal F)$ is either solvable in polynomial time or NP-complete.
\end{theorem}

Such coloring problems with forbidden cliques can be expressed in GMSNP, a logic capturing more general coloring problems. It was introduced independently in~\cite{OBDA,MMSNP2} and has been studied extensively in subsequent work~\cite{Containment,GoldenPath,MinimalFiniteFactors,ExtensionalESO,BarsukovMadelaine,ASNP, Bitter,Feller,ForbTournaments}. Our result makes progress on the question whether GMSNP admits a complexity dichotomy, raised in~\cite{OBDA}.
It also makes progress on the Bodirsky--Pinsker 
conjecture~\cite{Conjecture} predicting a P vs.\ NP-complete dichotomy for an even more general class of forbidden-pattern coloring problems.
Despite intense development of the theoretical tools to tackle such problems~\cite{CoresRamsey,SmoothApproximations,Wonderland,TopologicalBirkhoff,DecidabilityDefinability,AntoineSampling,PPDefinitions,ExistenceCores,CSPsOrbitReduction,DecidabilityInterpretability}, see also~\cite{BodirskyBook}, a proof of the general conjecture---while being verified for various subclasses~\cite{EqualityCSPs,TemporalCSPs, PhyloCSPs, PosetCSPs, MMSNP, GraphSAT, homogeneousGraphCSPs, SmoothApproximations, ForbTournaments,Bitter,Feller,HypergraphCSP, RCC5}---still seems to be out of reach.
One fundamental challenge lies in overcoming the problem of coloring with more than two ``relevant'' colors. This is exemplified by the fact that it proves to be a core challenge in a unified approach to the above dichotomies initiated in~\cite{SmoothApproximations}.
For further discussion of this challenge, we refer the reader to the concluding discussion in~\cite{Feller}.

\subsection{An upper bound for the complexity}\label{subsection:hat_functor}

The \emph{constraint satisfaction problem} associated to a relational structure $\str A$,\footnote{see Section~\ref{section:preliminaries} for precise definitions of the notions appearing in this section}
denoted by $\CSP(\str A)$ for short, is the computational problem of deciding whether a given structure $\str X$ admits a homomorphism to $\str A$.  
Each problem $\col(\mathcal F)$ admits a simple reduction to the CSP of a suitably chosen structure $\str A$, whose domain is the set $A$ of colors and whose relations encode the $\cat F$-free colorings of cliques.
Given a graph $G$, we construct a structure $\hat G$ in the same signature as $\str A$, whose domain is the set of edges of $G$ and whose relations encode the cliques contained in $G$.
The relations of both $\str A$ and $\hat G$ are chosen precisely, so that $\cat F$-free colorings of $G$
are exactly the homomorphisms
$\hat G\to \str A$. 
As such, the assignment $G \mapsto \hat G$ is a poly-time reduction from $\col(\cat F)$ to $\CSP(\str A)$.
Further,
CSPs admit a P vs.\ NP-complete dichotomy
\cite{Bulatov,ZhukFOCS,ZhukACM}, so in order to prove Theorem~\ref{thm:main} it suffices to show that $\col(\cat F)$ is NP-complete, whenever the corresponding $\CSP(\str A)$ is.
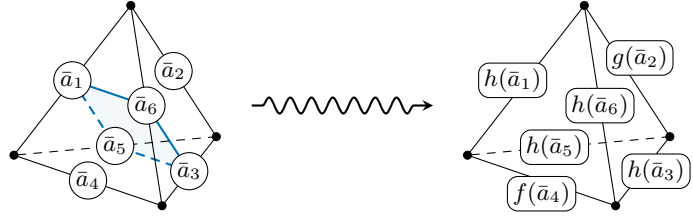
\begin{figure}
    \centering
    \definecolor{cbBlue}{HTML}{0072B2}
\definecolor{cbRed}{HTML}{D55E00}   
\definecolor{cbGreen}{HTML}{009E73}

\begin{tikzpicture}[scale=0.8]

    \begin{scope}[xshift=0cm]
        \tdplotsetmaincoords{70}{285}
        \begin{scope}[tdplot_main_coords, scale=1, line join=round, line cap=round]
            \tikzset{
                edgeLabel/.style={circle, fill=white, draw=black, inner sep=1pt, minimum size=14pt, font=\small},
                vertex/.style={circle, fill=black, draw=black, inner sep=0pt, minimum size=3pt}
            }
            \coordinate (Top)   at (0, 0, 2.828);
            \coordinate (Back)  at (0, 2, 0);
            \coordinate (Left)  at (-1.732, -1, 0);
            \coordinate (Right) at (1.732, -1, 0);
            \coordinate (M_LT) at (-0.866, -0.5, 1.414);
            \coordinate (M_BT) at (0, 1, 1.414);
            \coordinate (M_BR) at (0.866, 0.5, 0);
            \coordinate (M_LR) at (0, -1, 0);
            
            \draw[cbBlue, fill=cbBlue!15, fill opacity=0.3, draw opacity=0]
                (M_LT) -- (M_BT) -- (M_BR) -- (M_LR) -- cycle;
            \draw[cbBlue, thick] (M_LT) -- (M_LR);
            \draw[cbBlue, thick] (M_LT) -- (M_BT);
            \draw[cbBlue, thick, dashed] (M_LR) -- (M_BR);
            \draw[cbBlue, thick, dashed] (M_BT) -- (M_BR);
            
            \draw (Left) -- node[edgeLabel] {$\bar a_4$} (Back);
            \draw[dashed] (Right) -- node[edgeLabel, solid] {$\bar a_5$} (Back);
            \draw (Top) -- node[edgeLabel] {$\bar a_1$} (Back);
            \draw (Left) -- node[edgeLabel] {$\bar a_3$} (Right);
            \draw (Top) -- node[edgeLabel] {$\bar a_6$} (Left);
            \draw (Top) -- node[edgeLabel] {$\bar a_2$} (Right);

            \node[vertex] at (Top) {};
            \node[vertex] at (Back) {};
            \node[vertex] at (Left) {};
            \node[vertex] at (Right) {};
        \end{scope}
\end{scope}

    \draw[->, thick, >=stealth,
          decoration={snake, amplitude=3pt, segment length=8pt, pre length=4pt, post length=6pt},
          decorate]
        (2.0, 1) -- (5.0, 1);

    \begin{scope}[xshift=7.5cm]
        \tdplotsetmaincoords{70}{285}
        \begin{scope}[tdplot_main_coords, scale=1, line join=round, line cap=round]
            \tikzset{
                edgeLabel/.style={rectangle, rounded corners, fill=white, draw=black, inner sep=2pt, font=\small},
                vertex/.style={circle, fill=black, draw=black, inner sep=0pt, minimum size=3pt}
            }
            \coordinate (Top)   at (0, 0, 2.828);
            \coordinate (Back)  at (0, 2, 0);
            \coordinate (Left)  at (-1.732, -1, 0);
            \coordinate (Right) at (1.732, -1, 0);
            \draw (Left) -- node[edgeLabel, pos=0.5,yshift=-5] {$f(\bar a_4)$} (Back);
            \draw[dashed] (Right) -- node[edgeLabel, solid, pos=0.57] {$h(\bar a_5)$} (Back);
            \draw (Top) -- node[edgeLabel, pos=0.5, xshift=-5] {$h(\bar a_1)$} (Back);
            \draw (Left) -- node[edgeLabel, pos=0.5,xshift=5] {$h(\bar a_3)$} (Right);
            \draw (Top) -- node[edgeLabel, pos=0.5] {$h(\bar a_6)$} (Left);
            \draw (Top) -- node[edgeLabel, pos=0.4, xshift=8] {$g(\bar a_2)$} (Right);

            \node[vertex] at (Top) {};
            \node[vertex] at (Back) {};
            \node[vertex] at (Left) {};
            \node[vertex] at (Right) {};
        \end{scope}
    \end{scope}

\end{tikzpicture}
    \caption{Application of a cut-homotopy $h$ between $f, g$ to a tuple of colorings of $K_4$. The blue plane indicates the cut.}
    \label{fig:cutsarecool}
\end{figure}

\begin{exam}[\textsc{Edge-Mono-Tri}]
    Let $\cat F$ contain monochromatic triangles in two colors. The corresponding structure $\str A$ is boolean, with a ternary relation consisting of all non-monochromatic colorings of the triangle: 
    \begin{equation*}
        \{(0,0,1),(0,1,0),(1,0,0),(1,1,0),(1,0,1),(0,1,1)\}.
    \end{equation*}
    In this case, $\CSP(\str A)$ is the classical \textsc{NotAllEqual-SAT} problem.
\end{exam}

\subsection{Polymorphisms and cut-homotopies}
To every $\CSP(\str A)$,
one can associate
a structure $\pol(\str A)$, the \emph{polymorphism minion of $\str A$}, which determines its complexity up to log-space reductions~\cite{pcspBible}.
Polymorphisms are homomorphisms $\str A^n \to \str A$ and they allow combining solutions of $\CSP(\str A)$. Namely, given an $n$-tuple $\str X\to \str A^n$ of solutions to a given input $\str X$, composing with a polymorphism $\str A^n \to \str A$ creates a new solution $\str X\to \str A$. 
While the aforementioned results of Bulatov~\cite{Bulatov} and Zhuk~\cite{ZhukFOCS,ZhukACM} provide efficient algorithms in the presence of ``non-trivial'' polymorphisms of $\str A$, their absence readily implies hardness of $\CSP(\str A)$~\cite{BulatovJeavonsKrokhin}.
The latter statement can be formalized by requiring the existence of a \emph{minion homomorphism} $\pol(\str A) \to \id$, see \cite{pcspBible,wonderlandOfAdjunctions}.

In the context of edge-coloring problems, we introduce the notion of \emph{cut-homo\-topies} between polymorphisms, enriching the structure of the polymorphism minion $\pol(\str A)$.
Cut-homotopies allow for even further combinations of solutions to $\col(\cat F)$: Let $G$ be a graph and $\chi\colon\hat G\to \str A^n$ be a homomorphism, which corresponds to an $n$-tuple of $\cat F$-free colorings of $G$.
Given a cut-homotopy between two polymorphisms $f,g\colon\str A^n\to\str A$, every cut of $G$ into two pieces produces a new $\cat F$-free coloring of $G$ that agrees with $f \circ \chi$ on edges contained in one side of the cut, and with $g \circ \chi$ on the other side of the cut. The cut-homotopy provides a coherent way of coloring the cut edges, as visualized in Figure~\ref{fig:cutsarecool}.  
Using cut-homotopies, we prove a hardness criterion for $\col(\cat F)$, analogous to the hardness criterion for CSPs.

\begin{restatable}{theorem}{thmReduction}
\label{thm:reduction}
    Assume that there is a minion homomorphism $\alpha\colon \Pol(\str A)\to \id$ that collapses cut-homotopies, meaning $\alpha(f)=\alpha(g)$ whenever $f$ and $g$ are cut-homotopic polymorphisms of $\str A$. Then the problem $\col(\cat F)$ is NP-complete.
\end{restatable}

\subsection{Gadget reductions}

To prove Theorem~\ref{thm:reduction} we will reduce $\CSP(\str A)$ to $\col(\cat F)$
by producing, for every input $\str X$ to $\CSP(\str A)$, a graph $G_{\str X}$ that admits an $\cat F$-free coloring if and only if there is a homomorphism $\str X \to \str A$.
To do so, fix a graph $G_D$ and
for every $r$-ary relation $S$ of $\str A$, a graph $G_S$ containing $r$ marked copies of $G_D$.
Given $\str X$, we construct $G_{\str X}$ by taking a copy of $G_D$ for every $x \in X$ and for every tuple $s$ of every relation $S$ of $\str X$ a copy of $G_S$. 
For every $x \in X$ we then identify the corresponding copy of $G_D$ with the $i$th marked copy of $G_D$ in $G_S$ corresponding to $s\in S$ whenever $x$ appears in the $i$th entry of $s$.

By choosing suitable graphs $G_D$ and $G_S$
---therein lies the difficulty--- we ensure that colorings of $G_{\str X}$ translate to a homomorphism $\str X \to \str A$, with the value of $x$ being determined by the coloring of the corresponding copy of $G_D$, and vice versa.

\begin{exam}[\textsc{Vertex-Mono-Tri}]
    Consider the problem of finding a \emph{vertex}-coloring of a graph with two colors that
    avoids monochromatic triangles.
    Trying to reduce \textsc{NotAllEqual-SAT} to this problem, choose $G_D$ and $G_S$ to be $K_1$ and $K_3$ respectively. The resulting replacement procedure is displayed in Figure~\ref{fig:vertex_small_girth}.
    However, solutions of a given input $\str X$ of \textsc{NotAllEqual-SAT} do not always translate to valid colorings of the resulting graph, as new triangles might be created during the identification process.
\end{exam}
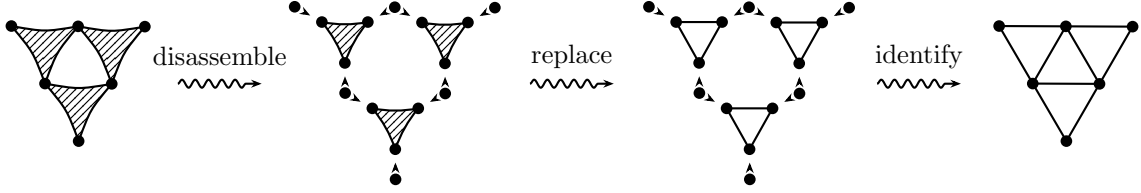
\begin{figure}
    \centering
\newcommand{\arrowlabelA}{disassemble}
\newcommand{\arrowlabelB}{replace}
\newcommand{\arrowlabelC}{identify}
 
\begin{tikzpicture}[scale=0.73,
  >={Stealth[length=4pt,width=3pt]},
  thick
]
 
\def\mygeom{
  \coordinate (TOP) at (-1.2124,  0.7000);
  \coordinate (BL)  at ( 0.0062, -1.3892);
  \coordinate (BR)  at ( 1.2062,  0.6892);
  \coordinate (ML)  at (-0.6031, -0.3446);
  \coordinate (MR)  at (-0.0031,  0.6946);
  \coordinate (MB)  at ( 0.6062, -0.3500);
  \coordinate (hT)  at (-0.6062,  0.3500);
  \coordinate (hL)  at ( 0.0031, -0.6946);
  \coordinate (hR)  at ( 0.6031,  0.3446);
}
 
\def\mygeomBig{
  \coordinate (TOP) at (-1.8187,  1.0500);
  \coordinate (BL)  at ( 0.0093, -2.0838);
  \coordinate (BR)  at ( 1.8093,  1.0338);
  \coordinate (ML)  at (-0.9047, -0.5169);
  \coordinate (MR)  at (-0.0047,  1.0419);
  \coordinate (MB)  at ( 0.9093, -0.5250);
  \coordinate (hT)  at (-0.9093,  0.5250);
  \coordinate (hL)  at ( 0.0047, -1.0419);
  \coordinate (hR)  at ( 0.9047,  0.5169);
}
 
\def\sone{0}
\def\stwo{5.74}
\def\sthree{12.16}
\def\sfour{17.90}
 
\begin{scope}[shift={(\sone,0)}]
  \mygeom
  \coordinate (ctrlT_TML) at (-0.8173,  0.2294);
  \coordinate (ctrlT_MLR) at (-0.3940,  0.2275);
  \coordinate (ctrlT_MRT) at (-0.6073,  0.5931);
 
  \coordinate (ctrlL_MLB) at (-0.2080, -0.8152);
  \coordinate (ctrlL_BLM) at ( 0.2153, -0.8171);
  \coordinate (ctrlL_MBM) at ( 0.0020, -0.4515);
 
  \coordinate (ctrlR_MRM) at ( 0.3920,  0.2240);
  \coordinate (ctrlR_MBR) at ( 0.8153,  0.2221);
  \coordinate (ctrlR_BRM) at ( 0.6020,  0.5877);
 
  \fill[pattern=north east lines]
    (TOP) .. controls (ctrlT_TML) .. (ML)
          .. controls (ctrlT_MLR) .. (MR)
          .. controls (ctrlT_MRT) .. cycle;
  \fill[pattern=north east lines]
    (ML) .. controls (ctrlL_MLB) .. (BL)
         .. controls (ctrlL_BLM) .. (MB)
         .. controls (ctrlL_MBM) .. cycle;
  \fill[pattern=north east lines]
    (MR) .. controls (ctrlR_MRM) .. (MB)
         .. controls (ctrlR_MBR) .. (BR)
         .. controls (ctrlR_BRM) .. cycle;
 
  \draw (TOP) .. controls (ctrlT_TML) .. (ML)
              .. controls (ctrlT_MLR) .. (MR)
              .. controls (ctrlT_MRT) .. cycle;
  \draw (ML)  .. controls (ctrlL_MLB) .. (BL)
              .. controls (ctrlL_BLM) .. (MB)
              .. controls (ctrlL_MBM) .. cycle;
  \draw (MR)  .. controls (ctrlR_MRM) .. (MB)
              .. controls (ctrlR_MBR) .. (BR)
              .. controls (ctrlR_BRM) .. cycle;
 
  \foreach \p in {TOP,BL,BR,ML,MR,MB} { \fill (\p) circle (3pt); }
\end{scope}
 
\draw[->, decorate, decoration={snake, amplitude=1.5pt, segment length=5pt, post length=4pt}] (1.8038,-0.347)--(3.3238,-0.347);
\node[font=\normalsize] at (2.5638,0.153) {\arrowlabelA};
 
\begin{scope}[shift={(\stwo,0)}]
  \mygeomBig

  \coordinate (vTOP) at (TOP);
  \coordinate (vBL)  at (BL);
  \coordinate (vBR)  at (BR);
  \coordinate (vML)  at (ML);
  \coordinate (vMR)  at (MR);
  \coordinate (vMB)  at (MB);

  \coordinate (hTtl) at (-1.3256,  0.7725);
  \coordinate (hTtr) at (-0.4931,  0.7725);
  \coordinate (hTap) at (-0.9093,  0.0300);
  \coordinate (cshT_tltr_a) at (-1.0392,  0.7180);
  \coordinate (cshT_tltr_b) at (-0.7795,  0.7180);
  \coordinate (cshT_trap_a) at (-0.6821,  0.5443);
  \coordinate (cshT_trap_b) at (-0.8119,  0.3126);
  \coordinate (cshT_aptl_a) at (-1.0067,  0.3126);
  \coordinate (cshT_aptl_b) at (-1.1366,  0.5443);

  \coordinate (hLtl) at (-0.4115, -0.7944);
  \coordinate (hLtr) at ( 0.4209, -0.7944);
  \coordinate (hLap) at ( 0.0047, -1.5369);
  \coordinate (cshL_tltr_a) at (-0.1252, -0.8488);
  \coordinate (cshL_tltr_b) at ( 0.1346, -0.8488);
  \coordinate (cshL_trap_a) at ( 0.2319, -1.0226);
  \coordinate (cshL_trap_b) at ( 0.1021, -1.2543);
  \coordinate (cshL_aptl_a) at (-0.0927, -1.2543);
  \coordinate (cshL_aptl_b) at (-0.2225, -1.0226);

  \coordinate (hRtl) at ( 0.4884,  0.7644);
  \coordinate (hRtr) at ( 1.3209,  0.7644);
  \coordinate (hRap) at ( 0.9047,  0.0219);
  \coordinate (cshR_tltr_a) at ( 0.7748,  0.7099);
  \coordinate (cshR_tltr_b) at ( 1.0345,  0.7099);
  \coordinate (cshR_trap_a) at ( 1.1319,  0.5362);
  \coordinate (cshR_trap_b) at ( 1.0021,  0.3045);
  \coordinate (cshR_aptl_a) at ( 0.8073,  0.3045);
  \coordinate (cshR_aptl_b) at ( 0.6774,  0.5362);

  \fill[pattern=north east lines]
    (hTtl) .. controls (cshT_tltr_a) and (cshT_tltr_b) .. (hTtr)
           .. controls (cshT_trap_a) and (cshT_trap_b) .. (hTap)
           .. controls (cshT_aptl_a) and (cshT_aptl_b) .. cycle;
  \fill[pattern=north east lines]
    (hLtl) .. controls (cshL_tltr_a) and (cshL_tltr_b) .. (hLtr)
           .. controls (cshL_trap_a) and (cshL_trap_b) .. (hLap)
           .. controls (cshL_aptl_a) and (cshL_aptl_b) .. cycle;
  \fill[pattern=north east lines]
    (hRtl) .. controls (cshR_tltr_a) and (cshR_tltr_b) .. (hRtr)
           .. controls (cshR_trap_a) and (cshR_trap_b) .. (hRap)
           .. controls (cshR_aptl_a) and (cshR_aptl_b) .. cycle;

  \draw (hTtl) .. controls (cshT_tltr_a) and (cshT_tltr_b) .. (hTtr)
               .. controls (cshT_trap_a) and (cshT_trap_b) .. (hTap)
               .. controls (cshT_aptl_a) and (cshT_aptl_b) .. cycle;
  \draw (hLtl) .. controls (cshL_tltr_a) and (cshL_tltr_b) .. (hLtr)
               .. controls (cshL_trap_a) and (cshL_trap_b) .. (hLap)
               .. controls (cshL_aptl_a) and (cshL_aptl_b) .. cycle;
  \draw (hRtl) .. controls (cshR_tltr_a) and (cshR_tltr_b) .. (hRtr)
               .. controls (cshR_trap_a) and (cshR_trap_b) .. (hRap)
               .. controls (cshR_aptl_a) and (cshR_aptl_b) .. cycle;

  \foreach \p in {vTOP,vBL,vBR,vML,vMR,vMB,
                  hTtl,hTtr,hTap,
                  hLtl,hLtr,hLap,
                  hRtl,hRtr,hRap} {
    \fill (\p) circle (3pt);
  }

  \tikzset{myarrow/.style={->, shorten >=5pt, shorten <=5pt}}

  \draw[myarrow] (vTOP) -- (hTtl);
  \draw[myarrow] (vMR)  -- (hTtr);
  \draw[myarrow] (vML)  -- (hTap);

  \draw[myarrow] (vML)  -- (hLtl);
  \draw[myarrow] (vMB)  -- (hLtr);
  \draw[myarrow] (vBL)  -- (hLap);

  \draw[myarrow] (vMR)  -- (hRtl);
  \draw[myarrow] (vBR)  -- (hRtr);
  \draw[myarrow] (vMB)  -- (hRap);

\end{scope}
 
\draw[->, decorate, decoration={snake, amplitude=1.5pt, segment length=5pt, post length=4pt}] (8.1853,-0.347)--(9.7053,-0.347);
\node[font=\normalsize] at (8.9453,0.153) {\arrowlabelB};
 
\begin{scope}[shift={(\sthree,0)}]
  \mygeomBig
 
  \coordinate (vTOP) at (TOP);
  \coordinate (vBL)  at (BL);
  \coordinate (vBR)  at (BR);
  \coordinate (vML)  at (ML);
  \coordinate (vMR)  at (MR);
  \coordinate (vMB)  at (MB);
 
  \coordinate (tTtl) at ($(hT) + (-0.4163, 0.2475)$);
  \coordinate (tTtr) at ($(hT) + ( 0.4163, 0.2475)$);
  \coordinate (tTap) at ($(hT) + ( 0.0000,-0.4950)$);
 
  \coordinate (tLtl) at ($(hL) + (-0.4163, 0.2475)$);
  \coordinate (tLtr) at ($(hL) + ( 0.4163, 0.2475)$);
  \coordinate (tLap) at ($(hL) + ( 0.0000,-0.4950)$);
 
  \coordinate (tRtl) at ($(hR) + (-0.4163, 0.2475)$);
  \coordinate (tRtr) at ($(hR) + ( 0.4163, 0.2475)$);
  \coordinate (tRap) at ($(hR) + ( 0.0000,-0.4950)$);
 
  \draw (tTtl)--(tTtr)--(tTap)--cycle;
  \draw (tLtl)--(tLtr)--(tLap)--cycle;
  \draw (tRtl)--(tRtr)--(tRap)--cycle;
 
  \foreach \p in {vTOP,vBL,vBR,vML,vMR,vMB,
                  tTtl,tTtr,tTap,
                  tLtl,tLtr,tLap,
                  tRtl,tRtr,tRap} {
    \fill (\p) circle (3pt);
  }
 
  \tikzset{myarrow/.style={->, shorten >=5pt, shorten <=5pt}}
 
  \draw[myarrow] (vTOP) -- (tTtl);
  \draw[myarrow] (vML)  -- (tTap);
  \draw[myarrow] (vMR)  -- (tTtr);
 
  \draw[myarrow] (vML)  -- (tLtl);
  \draw[myarrow] (vBL)  -- (tLap);
  \draw[myarrow] (vMB)  -- (tLtr);
 
  \draw[myarrow] (vMR)  -- (tRtl);
  \draw[myarrow] (vMB)  -- (tRap);
  \draw[myarrow] (vBR)  -- (tRtr);
 
\end{scope}
 
\draw[->, decorate, decoration={snake, amplitude=1.5pt, segment length=5pt, post length=4pt}] (14.4760,-0.347)--(15.9960,-0.347);
\node[font=\normalsize] at (15.2360,0.153) {\arrowlabelC};
 
\begin{scope}[shift={(\sfour,0)}]
  \mygeom
 
  \coordinate (vTOP) at (TOP);
  \coordinate (vBL)  at (BL);
  \coordinate (vBR)  at (BR);
  \coordinate (vML)  at (ML);
  \coordinate (vMR)  at (MR);
  \coordinate (vMB)  at (MB);
 
  \draw (vTOP) -- (vML) -- (vMR) -- cycle;
  \draw (vML)  -- (vBL) -- (vMB) -- cycle;
  \draw (vMR)  -- (vMB) -- (vBR) -- cycle;
 
  \foreach \p in {vTOP,vBL,vBR,vML,vMR,vMB} {
    \fill (\p) circle (3pt);
  }

\end{scope}
 
\end{tikzpicture}
    \caption{The
    gadget replacement may add implicit constraints.}
    \label{fig:vertex_small_girth}
\end{figure}
This issue was resolved for vertex-coloring problems, by showing that, without changing the computational complexity, one may assume that the inputs of a CSP have large \emph{girth}, meaning that they do not contain short cycles~\cite{FederVardi,Kun}. In this way, no new triangles are introduced and solutions of such inputs to \textsc{NotAllEqual-SAT} correspond to valid colorings of the graph obtained by the above replacement procedure. 
However, this approach fails for edge colorings, as the replacement procedure might create new cliques, even if inputs have large girth.

\begin{exam}
    Consider the problem \textsc{Edge-Mono-Tri} from above.
    Figure~\ref{fig:edge_large_girth} displays the na\"ive gadget replacement, taking $G_D$ to be an edge and $G_S$ to be a triangle, with the colors indicating the directions in which the edges are identified.
    Note that the identification step collapses some edges that are associated with different elements of the input $\str X$ (red-green and black-blue edges in Figure~\ref{fig:edge_large_girth}). In particular, solutions of $\str X$ cannot correspond to valid colorings of the resulting graph.
\end{exam}

To avoid this problem, we look for \emph{equality-free gadgets}, i.e., gadgets in which copies of $G_D$ within $G_S$ are pairwise disjoint. Moreover, if the distance between distinct copies of $G_D$ is large enough, no new cliques will be created by the replacement, even if the girth of the input is small. 
In this case however, $G_D$ cannot be a single edge in general, as the following example shows.

\begin{exam}[\textsc{LocalSwitches}]
    Let $A$ consist of two colors and let $\cat F$ be the family of forbidden edge-colorings of $K_4$, which are obtained from the two monochromatic colorings by switching the colors of all the edges incident to a specific vertex. Up to symmetries, there are six of them: 
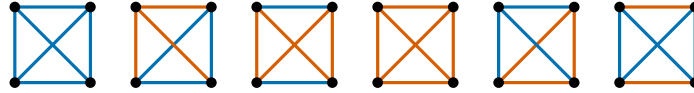
\begin{center}
\definecolor{cbBlue}{HTML}{0072B2}
\definecolor{cbRed}{HTML}{D55E00}   
\definecolor{cbGreen}{HTML}{009E73}

\newcommand{\blackedge}[2]{\draw[color=cbBlue, line width=1.2pt] (#1) -- (#2);
}
\newcommand{\rededge}[2]{\draw[color=cbRed, line width=1.2pt] (#1) -- (#2);
}
 
\newcommand{\kfouredge}[3]{\def\colorblack{black}\def\givencolor{#1}\ifx\givencolor\colorblack
    \draw[color=cbBlue, line width=1.2pt] (#2) -- (#3);
  \else
    \draw[color=cbRed, line width=1.2pt] (#2) -- (#3);
  \fi
}
 
\newcommand{\kfour}[7]{\begin{scope}[xshift=#1, line join=miter, line cap=rect]
    \kfouredge{#5}{1,0}{0,-1}
    \kfouredge{#4}{0,0}{1,-1}
    \kfouredge{#2}{0,0}{1,0}
    \kfouredge{#3}{0,0}{0,-1}
    \kfouredge{#6}{1,0}{1,-1}
    \kfouredge{#7}{0,-1}{1,-1}
    \foreach \x/\y in {0/0, 1/0, 0/-1, 1/-1} {
      \fill[black] (\x,\y) circle (2pt);
    }
  \end{scope}
}
 
\begin{tikzpicture}
  \kfour{0.0cm}{black}{black}{black}{black}{black}{black}
  \kfour{1.6cm}{BrickRed}{BrickRed}{BrickRed}{black}{black}{black}
  \kfour{3.2cm}{black}{BrickRed}{BrickRed}{BrickRed}{BrickRed}{black}
  \kfour{8.0cm}{BrickRed}{black}{black}{black}{black}{BrickRed}
  \kfour{6.4cm}{black}{black}{black}{BrickRed}{BrickRed}{BrickRed}
  \kfour{4.8cm}{BrickRed}{BrickRed}{BrickRed}{BrickRed}{BrickRed}{BrickRed}
\end{tikzpicture}
 
\end{center}
    Observe that, given any graph $G_S$ with an $\cat F$-free coloring, a new $\cat F$-free coloring can be obtained by switching the colors around any vertex of $G_S$.
    In particular, if $G_S$ has an $\cat F$-free coloring, then any pairwise disjoint edges $e_1,\dots,e_r$ can take any possible combination of the two colors.
\end{exam}

We extend the notion of cut-homotopy to colorings of graphs and, using Ramsey-theoretic tools, we construct graphs $G_D$ and $G_S$ with the following properties.
Every coloring of $G_D$ that extends to a coloring of $G_S$ is cut-homotopic to a coloring obtained by applying a polymorphism $f$ to a fixed tuple of colorings (Proposition~\ref{prop:homotopy_witness2}).
Moreover, $G_D$ can be chosen such that $f$ is unique up to cut-homotopy (Proposition~\ref{prop:homotopy_witness}).
In this sense, colorings of $G_D$ that extend to $G_S$ correspond, up to cut-homotopy, to polymorphisms of $\str A$.
Using the minion homomorphisms in the assumption of Theorem~\ref{thm:reduction}, we show that this gadget yields a reduction from $\CSP(\str A)$ to $\col(\cat F)$.

Cut-homotopies were inspired by the notion of \emph{multi-polymorphisms}, which characterize equality-free primitive-positive definability in relational structures~\cite{geiger1968closed,poschel2004galois}, which corresponds to equality-free gadget reductions between CSPs, where the domain gadget $G_D$ is a single vertex. 

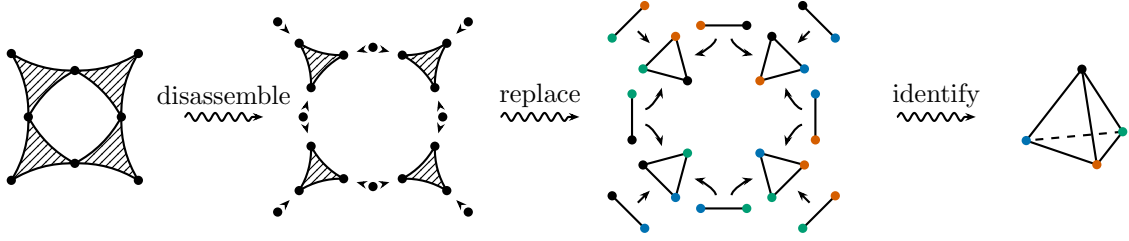
\begin{figure}
    \centering
\definecolor{cbBlue}{HTML}{0072B2}
\definecolor{cbRed}{HTML}{D55E00}   
\definecolor{cbGreen}{HTML}{009E73}

\newcommand{\arrowlabelA}{disassemble}
\newcommand{\arrowlabelB}{replace}
\newcommand{\arrowlabelC}{identify}

\newcommand{\vcSeSZeroa}{black}
\newcommand{\vcSeSZerob}{cbGreen}
\newcommand{\vcSeSOna}{cbRed}
\newcommand{\vcSeSONb}{black}
\newcommand{\vcSeSTwa}{cbBlue}
\newcommand{\vcSeSTwb}{cbRed}
\newcommand{\vcSeSTha}{cbGreen}
\newcommand{\vcSeSTHb}{cbBlue}
\newcommand{\vcSeOutTa}{cbGreen}
\newcommand{\vcSeOutTb}{cbRed}
\newcommand{\vcSeOutRa}{black}
\newcommand{\vcSeOutRb}{cbBlue}
\newcommand{\vcSeOutBa}{cbRed}
\newcommand{\vcSeOutBb}{cbGreen}
\newcommand{\vcSeOutLa}{cbBlue}
\newcommand{\vcSeOutLb}{black}
\newcommand{\vcStTva}{black}
\newcommand{\vcStTvb}{cbRed}
\newcommand{\vcStTvc}{cbGreen}
\newcommand{\vcStRva}{cbRed}
\newcommand{\vcStRvb}{cbBlue}
\newcommand{\vcStRvc}{black}
\newcommand{\vcStBva}{cbGreen}
\newcommand{\vcStBvb}{cbBlue}
\newcommand{\vcStBvc}{cbRed}
\newcommand{\vcStLva}{black}
\newcommand{\vcStLvb}{cbGreen}
\newcommand{\vcStLvc}{cbBlue}
\newcommand{\vcFourTop}{black}
\newcommand{\vcFourBack}{cbBlue}
\newcommand{\vcFourLeft}{cbRed}
\newcommand{\vcFourRight}{cbGreen}

\begin{tikzpicture}[scale=0.58,
  >={Stealth[length=4pt,width=3pt]},
  thick
]

\def\mygeom{
  \coordinate (S0)   at (-1.0607,  0.0000);
  \coordinate (S1)   at ( 0.0000,  1.0607);
  \coordinate (S2)   at ( 1.0607,  0.0000);
  \coordinate (S3)   at ( 0.0000, -1.0607);
  \coordinate (OutT) at (-1.4489,  1.4489);
  \coordinate (OutR) at ( 1.4489,  1.4489);
  \coordinate (OutB) at ( 1.4489, -1.4489);
  \coordinate (OutL) at (-1.4489, -1.4489);
}

\def\mygeomBig{
  \coordinate (S0)   at (-1.5910,  0.0000);
  \coordinate (S1)   at ( 0.0000,  1.5910);
  \coordinate (S2)   at ( 1.5910,  0.0000);
  \coordinate (S3)   at ( 0.0000, -1.5910);
  \coordinate (OutT) at (-2.1733,  2.1733);
  \coordinate (OutR) at ( 2.1733,  2.1733);
  \coordinate (OutB) at ( 2.1733, -2.1733);
  \coordinate (OutL) at (-2.1733, -2.1733);
  \coordinate (hTop) at (-1.2548,  1.2548);
  \coordinate (hRt)  at ( 1.2548,  1.2548);
  \coordinate (hBot) at ( 1.2548, -1.2548);
  \coordinate (hLt)  at (-1.2548, -1.2548);
}

\def\sone{0}
\def\stwo{6.80}
\def\sthree{14.80}
\def\sfour{22.80}

\begin{scope}[shift={(\sone,0)}]
  \mygeom

  \coordinate (cT_S0Out_a) at (-1.0452,  0.5218);
  \coordinate (cT_S0Out_b) at (-1.1746,  1.0047);
  \coordinate (cT_OutS1_a) at (-1.0047,  1.1746);
  \coordinate (cT_OutS1_b) at (-0.5218,  1.0452);
  \coordinate (cT_S1S0_a)  at (-0.4596,  0.8132);
  \coordinate (cT_S1S0_b)  at (-0.8132,  0.4596);

  \coordinate (cR_S1Out_a) at ( 0.5218,  1.0452);
  \coordinate (cR_S1Out_b) at ( 1.0047,  1.1746);
  \coordinate (cR_OutS2_a) at ( 1.1746,  1.0047);
  \coordinate (cR_OutS2_b) at ( 1.0452,  0.5218);
  \coordinate (cR_S2S1_a)  at ( 0.8132,  0.4596);
  \coordinate (cR_S2S1_b)  at ( 0.4596,  0.8132);

  \coordinate (cB_S2Out_a) at ( 1.0452, -0.5218);
  \coordinate (cB_S2Out_b) at ( 1.1746, -1.0047);
  \coordinate (cB_OutS3_a) at ( 1.0047, -1.1746);
  \coordinate (cB_OutS3_b) at ( 0.5218, -1.0452);
  \coordinate (cB_S3S2_a)  at ( 0.4596, -0.8132);
  \coordinate (cB_S3S2_b)  at ( 0.8132, -0.4596);

  \coordinate (cL_S3Out_a) at (-0.5218, -1.0452);
  \coordinate (cL_S3Out_b) at (-1.0047, -1.1746);
  \coordinate (cL_OutS0_a) at (-1.1746, -1.0047);
  \coordinate (cL_OutS0_b) at (-1.0452, -0.5218);
  \coordinate (cL_S0S3_a)  at (-0.8132, -0.4596);
  \coordinate (cL_S0S3_b)  at (-0.4596, -0.8132);

  \fill[pattern=north east lines]
    (S0) .. controls (cT_S0Out_a) and (cT_S0Out_b) .. (OutT)
         .. controls (cT_OutS1_a) and (cT_OutS1_b) .. (S1)
         .. controls (cT_S1S0_a)  and (cT_S1S0_b)  .. cycle;
  \fill[pattern=north east lines]
    (S1) .. controls (cR_S1Out_a) and (cR_S1Out_b) .. (OutR)
         .. controls (cR_OutS2_a) and (cR_OutS2_b) .. (S2)
         .. controls (cR_S2S1_a)  and (cR_S2S1_b)  .. cycle;
  \fill[pattern=north east lines]
    (S2) .. controls (cB_S2Out_a) and (cB_S2Out_b) .. (OutB)
         .. controls (cB_OutS3_a) and (cB_OutS3_b) .. (S3)
         .. controls (cB_S3S2_a)  and (cB_S3S2_b)  .. cycle;
  \fill[pattern=north east lines]
    (S3) .. controls (cL_S3Out_a) and (cL_S3Out_b) .. (OutL)
         .. controls (cL_OutS0_a) and (cL_OutS0_b) .. (S0)
         .. controls (cL_S0S3_a)  and (cL_S0S3_b)  .. cycle;

  \draw (S0) .. controls (cT_S0Out_a) and (cT_S0Out_b) .. (OutT)
             .. controls (cT_OutS1_a) and (cT_OutS1_b) .. (S1)
             .. controls (cT_S1S0_a)  and (cT_S1S0_b)  .. cycle;
  \draw (S1) .. controls (cR_S1Out_a) and (cR_S1Out_b) .. (OutR)
             .. controls (cR_OutS2_a) and (cR_OutS2_b) .. (S2)
             .. controls (cR_S2S1_a)  and (cR_S2S1_b)  .. cycle;
  \draw (S2) .. controls (cB_S2Out_a) and (cB_S2Out_b) .. (OutB)
             .. controls (cB_OutS3_a) and (cB_OutS3_b) .. (S3)
             .. controls (cB_S3S2_a)  and (cB_S3S2_b)  .. cycle;
  \draw (S3) .. controls (cL_S3Out_a) and (cL_S3Out_b) .. (OutL)
             .. controls (cL_OutS0_a) and (cL_OutS0_b) .. (S0)
             .. controls (cL_S0S3_a)  and (cL_S0S3_b)  .. cycle;

  \foreach \p in {S0,S1,S2,S3,OutT,OutR,OutB,OutL} { \fill (\p) circle (3pt); }
\end{scope}

\draw[->, decorate, decoration={snake, amplitude=1.5pt, segment length=5pt, post length=4pt}]
  (2.50, 0.0) -- (4.30, 0.0);
\node[font=\normalsize] at (3.40, 0.5) {\arrowlabelA};

\begin{scope}[shift={(\stwo,0)}]
  \mygeomBig

  \coordinate (vS0)   at (S0);
  \coordinate (vS1)   at (S1);
  \coordinate (vS2)   at (S2);
  \coordinate (vS3)   at (S3);
  \coordinate (vOutT) at (OutT);
  \coordinate (vOutR) at (OutR);
  \coordinate (vOutB) at (OutB);
  \coordinate (vOutL) at (OutL);

  \coordinate (tTbl) at (-1.4117,  0.6692);
  \coordinate (tTbr) at (-0.6692,  1.4117);
  \coordinate (tTap) at (-1.6835,  1.6835);

  \coordinate (tRtl) at ( 0.6692,  1.4117);
  \coordinate (tRbl) at ( 1.4117,  0.6692);
  \coordinate (tRap) at ( 1.6835,  1.6835);

  \coordinate (tBtl) at ( 0.6692, -1.4117);
  \coordinate (tBtr) at ( 1.4117, -0.6692);
  \coordinate (tBap) at ( 1.6835, -1.6835);

  \coordinate (tLtr) at (-1.4117, -0.6692);
  \coordinate (tLbr) at (-0.6692, -1.4117);
  \coordinate (tLap) at (-1.6835, -1.6835);

  \coordinate (csT_blbr_a) at (-1.2384,  0.9909);
  \coordinate (csT_blbr_b) at (-0.9909,  1.2384);
  \coordinate (csT_brap_a) at (-1.0345,  1.4009);
  \coordinate (csT_brap_b) at (-1.3726,  1.4915);
  \coordinate (csT_apbl_a) at (-1.4915,  1.3726);
  \coordinate (csT_apbl_b) at (-1.4009,  1.0345);

  \coordinate (csR_blbr_a) at ( 0.9909,  1.2384);
  \coordinate (csR_blbr_b) at ( 1.2384,  0.9909);
  \coordinate (csR_brap_a) at ( 1.4009,  1.0345);
  \coordinate (csR_brap_b) at ( 1.4915,  1.3726);
  \coordinate (csR_apbl_a) at ( 1.3726,  1.4915);
  \coordinate (csR_apbl_b) at ( 1.0345,  1.4009);

  \coordinate (csB_blbr_a) at ( 0.9909, -1.2384);
  \coordinate (csB_blbr_b) at ( 1.2384, -0.9909);
  \coordinate (csB_brap_a) at ( 1.4009, -1.0345);
  \coordinate (csB_brap_b) at ( 1.4915, -1.3726);
  \coordinate (csB_apbl_a) at ( 1.3726, -1.4915);
  \coordinate (csB_apbl_b) at ( 1.0345, -1.4009);

  \coordinate (csL_blbr_a) at (-1.2384, -0.9909);
  \coordinate (csL_blbr_b) at (-0.9909, -1.2384);
  \coordinate (csL_brap_a) at (-1.0345, -1.4009);
  \coordinate (csL_brap_b) at (-1.3726, -1.4915);
  \coordinate (csL_apbl_a) at (-1.4915, -1.3726);
  \coordinate (csL_apbl_b) at (-1.4009, -1.0345);

  \fill[pattern=north east lines]
    (tTbl) .. controls (csT_blbr_a) and (csT_blbr_b) .. (tTbr)
           .. controls (csT_brap_a) and (csT_brap_b) .. (tTap)
           .. controls (csT_apbl_a) and (csT_apbl_b) .. cycle;
  \fill[pattern=north east lines]
    (tRtl) .. controls (csR_blbr_a) and (csR_blbr_b) .. (tRbl)
           .. controls (csR_brap_a) and (csR_brap_b) .. (tRap)
           .. controls (csR_apbl_a) and (csR_apbl_b) .. cycle;
  \fill[pattern=north east lines]
    (tBtl) .. controls (csB_blbr_a) and (csB_blbr_b) .. (tBtr)
           .. controls (csB_brap_a) and (csB_brap_b) .. (tBap)
           .. controls (csB_apbl_a) and (csB_apbl_b) .. cycle;
  \fill[pattern=north east lines]
    (tLtr) .. controls (csL_blbr_a) and (csL_blbr_b) .. (tLbr)
           .. controls (csL_brap_a) and (csL_brap_b) .. (tLap)
           .. controls (csL_apbl_a) and (csL_apbl_b) .. cycle;

  \draw (tTbl) .. controls (csT_blbr_a) and (csT_blbr_b) .. (tTbr)
               .. controls (csT_brap_a) and (csT_brap_b) .. (tTap)
               .. controls (csT_apbl_a) and (csT_apbl_b) .. cycle;
  \draw (tRtl) .. controls (csR_blbr_a) and (csR_blbr_b) .. (tRbl)
               .. controls (csR_brap_a) and (csR_brap_b) .. (tRap)
               .. controls (csR_apbl_a) and (csR_apbl_b) .. cycle;
  \draw (tBtl) .. controls (csB_blbr_a) and (csB_blbr_b) .. (tBtr)
               .. controls (csB_brap_a) and (csB_brap_b) .. (tBap)
               .. controls (csB_apbl_a) and (csB_apbl_b) .. cycle;
  \draw (tLtr) .. controls (csL_blbr_a) and (csL_blbr_b) .. (tLbr)
               .. controls (csL_brap_a) and (csL_brap_b) .. (tLap)
               .. controls (csL_apbl_a) and (csL_apbl_b) .. cycle;

  \foreach \p in {vS0,vS1,vS2,vS3,vOutT,vOutR,vOutB,vOutL,
                  tTbl,tTbr,tTap,
                  tRtl,tRbl,tRap,
                  tBtl,tBtr,tBap,
                  tLtr,tLbr,tLap} {
    \fill (\p) circle (3pt);
  }

  \tikzset{myarrow/.style={->, shorten >=5pt, shorten <=5pt}}

  \draw[myarrow] (vS0)   -- (tTbl);
  \draw[myarrow] (vS1)   -- (tTbr);
  \draw[myarrow] (vOutT) -- (tTap);

  \draw[myarrow] (vS1)   -- (tRtl);
  \draw[myarrow] (vS2)   -- (tRbl);
  \draw[myarrow] (vOutR) -- (tRap);

  \draw[myarrow] (vS2)   -- (tBtr);
  \draw[myarrow] (vS3)   -- (tBtl);
  \draw[myarrow] (vOutB) -- (tBap);

  \draw[myarrow] (vS3)   -- (tLbr);
  \draw[myarrow] (vS0)   -- (tLtr);
  \draw[myarrow] (vOutL) -- (tLap);

\end{scope}

\draw[->, decorate, decoration={snake, amplitude=1.5pt, segment length=5pt, post length=4pt}]
  (9.71, 0.0) -- (11.51, 0.0);
\node[font=\normalsize] at (10.61, 0.5) {\arrowlabelB};
 
\begin{scope}[shift={(\sthree,0)}]
  \mygeomBig

  \coordinate (tTv0) at (-0.8261,  0.8261);
  \coordinate (tTv1) at (-1.0979,  1.8404);
  \coordinate (tTv2) at (-1.8404,  1.0979);
  \coordinate (tRv0) at ( 0.8261,  0.8261);
  \coordinate (tRv1) at ( 1.8404,  1.0979);
  \coordinate (tRv2) at ( 1.0979,  1.8404);
  \coordinate (tBv0) at ( 1.0979, -1.8404);
  \coordinate (tBv1) at ( 0.8261, -0.8261);
  \coordinate (tBv2) at ( 1.8404, -1.0979);
  \coordinate (tLv0) at (-1.8404, -1.0979);
  \coordinate (tLv1) at (-0.8261, -0.8261);
  \coordinate (tLv2) at (-1.0979, -1.8404);

  \coordinate (tTm20) at (-1.3333,  0.9620);
  \coordinate (tTm01) at (-0.9620,  1.3333);
  \coordinate (tTm12) at (-1.4691,  1.4691);
  \coordinate (tRm20) at ( 0.9620,  1.3333);
  \coordinate (tRm01) at ( 1.3333,  0.9620);
  \coordinate (tRm12) at ( 1.4691,  1.4691);
  \coordinate (tBm01) at ( 0.9620, -1.3333);
  \coordinate (tBm12) at ( 1.3333, -0.9620);
  \coordinate (tBm20) at ( 1.4691, -1.4691);
  \coordinate (tLm01) at (-1.3333, -0.9620);
  \coordinate (tLm12) at (-0.9620, -1.3333);
  \coordinate (tLm20) at (-1.4691, -1.4691);

  \coordinate (eS0a) at (-2.0910, -0.5251);
  \coordinate (eS0b) at (-2.0910,  0.5251);
  \coordinate (eS1a) at (-0.5251,  2.0910);
  \coordinate (eS1b) at ( 0.5251,  2.0910);
  \coordinate (eS2a) at ( 2.0910,  0.5251);
  \coordinate (eS2b) at ( 2.0910, -0.5251);
  \coordinate (eS3a) at ( 0.5251, -2.0910);
  \coordinate (eS3b) at (-0.5251, -2.0910);
  \coordinate (eOutTa) at (-2.5446,  1.8020);
  \coordinate (eOutTb) at (-1.8020,  2.5446);
  \coordinate (eOutRa) at ( 1.8020,  2.5446);
  \coordinate (eOutRb) at ( 2.5446,  1.8020);
  \coordinate (eOutBa) at ( 2.5446, -1.8020);
  \coordinate (eOutBb) at ( 1.8020, -2.5446);
  \coordinate (eOutLa) at (-1.8020, -2.5446);
  \coordinate (eOutLb) at (-2.5446, -1.8020);

  \coordinate (mS0)   at (-2.0910,  0.0000);
  \coordinate (mS1)   at ( 0.0000,  2.0910);
  \coordinate (mS2)   at ( 2.0910,  0.0000);
  \coordinate (mS3)   at ( 0.0000, -2.0910);
  \coordinate (mOutT) at (-2.1733,  2.1733);
  \coordinate (mOutR) at ( 2.1733,  2.1733);
  \coordinate (mOutB) at ( 2.1733, -2.1733);
  \coordinate (mOutL) at (-2.1733, -2.1733);

  \draw (tTv0)--(tTv1)--(tTv2)--cycle;
  \draw (tRv0)--(tRv1)--(tRv2)--cycle;
  \draw (tBv0)--(tBv1)--(tBv2)--cycle;
  \draw (tLv0)--(tLv1)--(tLv2)--cycle;

  \draw (eS0a)--(eS0b);
  \draw (eS1a)--(eS1b);
  \draw (eS2a)--(eS2b);
  \draw (eS3a)--(eS3b);
  \draw (eOutTa)--(eOutTb);
  \draw (eOutRa)--(eOutRb);
  \draw (eOutBa)--(eOutBb);
  \draw (eOutLa)--(eOutLb);

  \fill[color=\vcSeSZeroa]   (eS0a)   circle (3pt);
  \fill[color=\vcSeSZerob]   (eS0b)   circle (3pt);
  \fill[color=\vcSeSOna]   (eS1a)   circle (3pt);
  \fill[color=\vcSeSONb]   (eS1b)   circle (3pt);
  \fill[color=\vcSeSTwa]   (eS2a)   circle (3pt);
  \fill[color=\vcSeSTwb]   (eS2b)   circle (3pt);
  \fill[color=\vcSeSTha]   (eS3a)   circle (3pt);
  \fill[color=\vcSeSTHb]   (eS3b)   circle (3pt);
  \fill[color=\vcSeOutTa] (eOutTa) circle (3pt);
  \fill[color=\vcSeOutTb] (eOutTb) circle (3pt);
  \fill[color=\vcSeOutRa] (eOutRa) circle (3pt);
  \fill[color=\vcSeOutRb] (eOutRb) circle (3pt);
  \fill[color=\vcSeOutBa] (eOutBa) circle (3pt);
  \fill[color=\vcSeOutBb] (eOutBb) circle (3pt);
  \fill[color=\vcSeOutLa] (eOutLa) circle (3pt);
  \fill[color=\vcSeOutLb] (eOutLb) circle (3pt);
  \fill[color=\vcStTva] (tTv0) circle (3pt);
  \fill[color=\vcStTvb] (tTv1) circle (3pt);
  \fill[color=\vcStTvc] (tTv2) circle (3pt);
  \fill[color=\vcStRva] (tRv0) circle (3pt);
  \fill[color=\vcStRvb] (tRv1) circle (3pt);
  \fill[color=\vcStRvc] (tRv2) circle (3pt);
  \fill[color=\vcStBva] (tBv0) circle (3pt);
  \fill[color=\vcStBvb] (tBv1) circle (3pt);
  \fill[color=\vcStBvc] (tBv2) circle (3pt);
  \fill[color=\vcStLva] (tLv0) circle (3pt);
  \fill[color=\vcStLvb] (tLv1) circle (3pt);
  \fill[color=\vcStLvc] (tLv2) circle (3pt);

  \tikzset{myarrow/.style={->, shorten >=5pt, shorten <=5pt}}

  \draw[myarrow] (-2.0910, 0.1800) .. controls (-1.6555, 0.1800) and (-1.4443, 0.5480) .. (tTm20);
  \draw[myarrow] (-2.0910,-0.1800) .. controls (-1.6555,-0.1800) and (-1.4443,-0.5480) .. (tLm01);

  \draw[myarrow] (-0.1800, 2.0910) .. controls (-0.1800, 1.6555) and (-0.5480, 1.4443) .. (tTm01);
  \draw[myarrow] ( 0.1800, 2.0910) .. controls ( 0.1800, 1.6555) and ( 0.5480, 1.4443) .. (tRm20);

  \draw[myarrow] ( 2.0910, 0.1800) .. controls ( 1.6555, 0.1800) and ( 1.4443, 0.5480) .. (tRm01);
  \draw[myarrow] ( 2.0910,-0.1800) .. controls ( 1.6555,-0.1800) and ( 1.4443,-0.5480) .. (tBm12);

  \draw[myarrow] ( 0.1800,-2.0910) .. controls ( 0.1800,-1.6555) and ( 0.5480,-1.4443) .. (tBm01);
  \draw[myarrow] (-0.1800,-2.0910) .. controls (-0.1800,-1.6555) and (-0.5480,-1.4443) .. (tLm12);

  \draw[myarrow] (mOutT) -- (tTm12);
  \draw[myarrow] (mOutR) -- (tRm12);
  \draw[myarrow] (mOutB) -- (tBm20);
  \draw[myarrow] (mOutL) -- (tLm20);

\end{scope}

\draw[->, decorate, decoration={snake, amplitude=1.5pt, segment length=5pt, post length=4pt}]
  (18.75, 0.0) -- (20.55, 0.0);
\node[font=\normalsize] at (19.65, 0.5) {\arrowlabelC};
 
\begin{scope}[shift={(\sfour,0)}]
  \coordinate (TTop)   at ( 0.1702,  1.0908);
  \coordinate (TBack)  at (-1.1000, -0.5400);
  \coordinate (TLeft)  at ( 0.5105, -1.0908);
  \coordinate (TRight) at ( 1.1000, -0.3384);

  \draw (TLeft)--(TBack);
  \draw (TTop)--(TBack);
  \draw (TLeft)--(TRight);
  \draw (TTop)--(TLeft);
  \draw (TTop)--(TRight);

  \draw[dashed] (TRight)--(TBack);

  \fill[color=\vcFourTop]   (TTop)   circle (3pt);
  \fill[color=\vcFourBack]  (TBack)  circle (3pt);
  \fill[color=\vcFourLeft]  (TLeft)  circle (3pt);
  \fill[color=\vcFourRight] (TRight) circle (3pt);
\end{scope}

\end{tikzpicture}
    \caption{Since two edges share a vertex, the gadget replacement shrinks down the girth.}
    \label{fig:edge_large_girth}
\end{figure}

\subsection{Outlook}
We define cut-homotopies only for
those
problems discussed in the present paper, however they seem to naturally be applicable to more general forbidden-pattern-coloring problems, including edge-coloring problems with more general shapes of forbidden graphs, and also those that require coloring of ordered edges of potentially higher arity.
We wonder how far our techniques can scale, in particular: how much closer can cut-homotopies bring us to a  complexity dichotomy for GMSNP? 

Problems in the scope of the Bodirsky--Pinsker conjecture admit a standard reduction to a finite CSP~\cite{CSPsOrbitReduction}, similar to what we describe in Section~\ref{subsection:hat_functor} and the notion of cut-homotopies
can also be translated to this setting. However, there are tractable problems where the corresponding CSP is NP-complete; such problems include temporal~\cite{TemporalCSPs} and phylogeny~\cite{PhyloCSPs} constraint satisfaction.
We wonder where and how exactly our proof techniques fail for these problems,
but also if it is still possible to extract meaningful information
using our methods, perhaps similar to the treatment of temporal constraint satisfaction problems in~\cite{SmoothApproximations}.

\subsection{Overview}
In Section~\ref{section:preliminaries}, we recall CSPs and minions, and formally describe the structure $\str A$ and the reduction $G\mapsto\hat G$ corresponding to a family $\cat F$. 
In Section~\ref{section:cut_homotopies}, we define cut-homotopies of polymorphisms and an analogous notion for colorings of graphs. 
The reduction of Theorem~\ref{thm:reduction} is presented in Section~\ref{sec:reductions}.
In Section~\ref{sec:therearenoweirdhomotopies}, we show that the assumptions of Theorem~\ref{thm:reduction} are satisfied for a problem $\col(\cat F)$, whenever the corresponding CSP is NP-complete, completing the proof of Theorem~\ref{thm:main}.
Finally, in Section~\ref{sec:ramseysToolbox}, we prove the Ramsey-theoretic claims used in Sections~\ref{sec:reductions} and~\ref{sec:therearenoweirdhomotopies}.

\section{Preliminaries}\label{section:preliminaries}

\subsection{Graphs, colorings and the hat}\label{sec:graphcolhat}

In this paper, all graphs are assumed to be undirected and finite.
For a graph $G$, we will abuse the notation and refer to its vertex set as $G$ and denote its edges as $vw$, where $v,w \in G$. Since edges are symmetric, we have $vw =wv$.
For a natural number $n$, the clique $K_n$ is the graph with $n$ vertices and an edge $vw$ for each pair of distinct vertices $v,w\in K_n$. We identify its set of edges with $\binom{n}{2}$.

A relational \emph{structure} is a tuple $\str A = (A, R_1,\dots, R_q)$---where the \emph{domain} $A$ is a set and every \emph{relation} $R_i\subseteq A^{r_i}$ is a set of functions $r_i \to A$, for some finite set $r_i$ called the \emph{arity} of $R_i$. 
The tuple of arities $(r_1,\ldots, r_q)$ is the \emph{signature} of $\str A$.
Given another structure $\str X=(X,S_1, \dots, S_q)$ of the same signature, a map $f\colon X \to A$ is called a \emph{homomorphism} if $f\circ s\in R_i$ for every $s\in S_i$ and $i\leq q$.

The \emph{hat} associated to a fixed family $\cat F$ of edge-colored graphs assigns to every graph $G$ a structure $\hat G$ that is defined as follows.
The domain of $\hat G$ consists of the edges of $G$.
For every $n$ such that $\cat F$ contains at least one clique of size $n$, the structure $\hat G$ has a $\binom{n}{2}$-ary relation consisting of all maps of the form 
$$
\hat \phi \colon \binom{n}{2} \to \hat G, \quad ij \mapsto \phi(i)\phi(j),
$$ 
where $\phi$ is a homomorphism from $K_n$ to $G$. 
Generally, for a homomorphism of graphs $\phi\colon G \to H$, the map $\hat \phi\colon \hat G \to \hat H$ defined as $vw \mapsto \phi(v) \phi(w)$, is a homomorphism of structures. The hat preserves identity maps and composition, so it is a functor from the category of graphs to the category of structures in the signature dictated by $\cat F$.

We denote by $\str A$ the structure of the same signature whose domain is the set $A$ of colors that appear among the colorings of cliques in $\cat F$, where every $\binom{n}{2}$-ary relation $R_n$ consists of all $\cat F$-free edge-colorings of $K_n$.
Observe that a map from the set of edges of $G$ to $A$ is a homomorphism $\chi\colon \hat G\to \str A$, if and only if it is an $\cat F$-free coloring.
Further, a map $f\colon \str A^r \to \str A$ is a homomorphism
if and only if for every clique $K_n$ and every coloring $\chi \colon \hat K_n \to \str A^r$, the composition $f\circ \chi$ is an $\cat F$-free coloring of $K_n$.

\subsection{CSPs, polymorphisms and minions}
\label{sec:PrelimsCSPs}
The \emph{constraint satisfaction problem} $\CSP(\str A)$ of a relational structure $\str A$ takes structures $\str X$ of the same signature as $\str A$ as inputs and asks whether there is a homomorphism $\str X\to\str A$.
For a finite set $r$, $\str {\str A}^r$ denotes the $r$-th cartesian power of $\str A$ whose domain consists of functions $r\to A$. 
If $r = \{1,\dots,n\}$ we treat ${\str A}^r$ and $\str A^n$ as being equal.
Homomorphisms $\str A^r\to\str A$ are called $r$-ary \emph{polymorphisms}.
We denote the set of homomorphisms from $\str B$ to $\str A$ by $\hom(\str B, \str A)$.
The identity homomorphism of a structure $\str A$ is denoted by $\id_{\str A}$.

A \emph{minion} $\cat M$ is a functor from the category of finite sets to itself.
Explicitly, for every finite set $r$ there is a finite set $\cat M^{(r)}$ and for every map $\sigma\colon r\to s$ there is a map 
$
\cat M^{(\sigma)}\colon \cat M^{(r)}\to \cat M^{(s)}
$
respecting identity maps and composition. 
By $\id$, we denote the identity minion with $\id^{(r)} = r$.
A \emph{minion homomorphism} $\alpha \colon \cat M \to \cat N$ is a natural transformation, that is, $\alpha$ is a family of maps $\alpha_r\colon \cat M^{(r)} \to \cat N^{(r)}$ such that the following square commutes for every map $\sigma\colon r \to s$.
\begin{equation}
\label{eq:naturality}
    \begin{tikzcd}
    \cat M^{(r)} \arrow[r, "\cat M^{(\sigma)}"] \arrow[d, "\alpha_r"'] & \cat M^{(s)} \arrow[d, "\alpha_s"]\\
    \cat N^{(r)} \arrow[r, "\cat N^{(\sigma)}"] & \cat N^{(s)}
\end{tikzcd}
\end{equation}
We will frequently abuse notation and write $\alpha(x)$ instead of $\alpha_r(x)$.
The fundamental example of a minion is a \emph{polymorphism minion}: For a relational structure $\str A$, let $\pol(\str A)$ be the minion with
$
    \pol(\str A)^{(r)} = \hom(\str A^r,\str A).
$
Moreover, for a map $\sigma\colon r\to s$, the map $\pol(\str A)^{(\sigma)}$ assigns to each $f\colon \str A^r\to \str A$, the polymorphism 
$f\circ\sigma^*\colon\str A^s\to \str A$, $a \mapsto f(a \circ \sigma)$, where $\sigma^*$ denotes the precomposition map $a\mapsto a\circ\sigma$.
Of particular importance are minion homomorphisms $\alpha\colon\pol(\str A)\to\id$, as their existence implies hardness of $\CSP(\str A)$, see Corollary~9.5 in~\cite{pcspBible}. The naturality condition~(\ref{eq:naturality}) of such $\alpha$ precisely states that
$
\alpha( f\circ \sigma^* ) = \sigma(\alpha(f))  
$
for all polymorphisms $f$ and maps $\sigma$.

\section{Cut-homotopies}\label{section:cut_homotopies}

Consider the looped path $P_k$ of length $k$, i.e., the graph with vertices $0,\dots,k$ that has a loop $ii$ for every $i\in P_k$ and an edge $i(i+1)$ for each $0\leq i < k$.

\begin{definition}
We call a homomorphism $h\colon \str A^r \times \hat P_k\to \str A$ a $k$\emph{-step cut-homotopy of polymorphisms}. 
\end{definition}

Observe that every $h_{i}:= h(-,ii)\colon \str A^r\to\str A$ is a polymorphism of $\str A$; we say that $h_0$ and $h_k$ are $k$\emph{-step cut-homotopic} and write $h_{0}\sim_k h_{k}$.
Given a $k$-step cut-homotopy $h\colon \str A^r \times \hat P_k\to \str A$, for $0\leq i\leq k$ we shorten $h(-,ii)$ to $h_i$ and for $0\leq i < k$ we shorten $h(-,i(i+1))$ to $h_{i(i+1)}$. 
For $k=1$, we simply say \emph{cut-homotopy} and \emph{cut-homotopic}.

We refer to any graph homomorphism $G \to P_1$ as a \emph{cut of $G$}.
One can think of a cut partitioning the vertices of $G$ into two sets. Given a partition of $G$ into $G_0,G_1$ we denote by $G_0\mid G_1$ the cut $c: G \to P_1$ such that $c^{-1}(0) = G_0$ and $c^{-1}(1) = G_1$.
To illustrate their application, consider a cut-homotopy $h \colon \str A^r \times \hat P_1 \to \str A$, a graph $G$ and a coloring $\chi \colon \hat G \to \str A^r$. The polymorphisms $h_{0}, h_{1}$ induce two colorings of $G$, $h_{0}\circ \chi$ and $h_{1}\circ \chi$, and $h$ provides a way to ``translate'' between the two. More precisely, given \emph{any} cut $c \colon G \to P_1$, the coloring
\begin{equation*}
    \hat G\xrightarrow{(\chi,\hat c)} \str A^r\times\hat P_1\xrightarrow{h} \str A
\end{equation*} 
agrees with $h_0\circ \chi$ on all edges $vw$ with $\hat c(vw) = 00$, with $h_1 \circ \chi$ on all edges $vw$ with $\hat c(vw) = 11$ and with $h_{01}\circ \chi$ on the remaining edges.
For two graphs $G,H$, denote by $G\times H$ the graph whose vertices are all the pairs $(v,i)$ for $v\in G, i\in H$ and $(v,i)(w,j)$ is an edge if and only if $vw$ is an edge of $G$ and $ij$ is an edge of $H$.
If $vw \in \hat G$, $ij \in \hat H$ we get two different edges of $G \times H$ which, following our naming convention for edges of graphs, technically are $(v,i)(w,j)$ and $(v,j)(w,i)$. 
To have a slimmer notation however we will refer to them as $(vw,ij)$ and $(vw,ji)$ respectively.
\begin{definition}
We call a coloring $\eta\colon\widehat{G\times P_k}\to \str A$ a $k$\emph{-step cut-homotopy of colorings}, say that $\eta_0\coloneq\eta(- ,00)$ and $\eta_k\coloneq\eta(-,kk)$ are $k$\emph{-step cut-homotopic} and write $\eta_0\sim_k \eta_k$.
Again, for $k=1$, we simply say \emph{cut-homotopy} and \emph{cut-homotopic}.
\end{definition}
\begin{exam}
Suppose $\cat F$ consists of both monochromatic triangles in two colors. Figure~\ref{fig:cut-homotopy-graph-colorings} depicts a cut-homotopy between two $\cat F$-free colorings of the triangle $K_3$.
\end{exam}
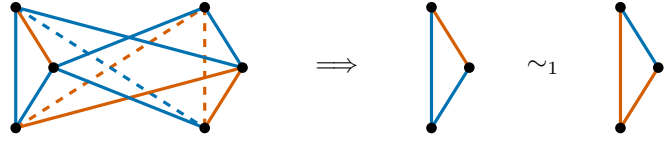
\begin{figure}
    \centering
\definecolor{cbBlue}{HTML}{0072B2}
\definecolor{cbRed}{HTML}{D55E00}   
\definecolor{cbGreen}{HTML}{009E73}

\begin{tikzpicture}

\newcommand{\blackedge}[2]{\draw[color=cbBlue, line width=1.2pt] (#1) -- (#2);
}
\newcommand{\rededge}[2]{\draw[color=cbRed, line width=1.2pt] (#1) -- (#2);
}

\newcommand{\blackedged}[2]{\draw[color=cbBlue, line width=1.2pt,dashed] (#1) -- (#2);
}
\newcommand{\rededged}[2]{\draw[color=cbRed, line width=1.2pt,dashed] (#1) -- (#2);
}
\coordinate (v0a) at (-1.0,  0.8);
\coordinate (v2a) at (-0.5,  0);
\coordinate (v1a) at (-1.0, -0.8);

\coordinate (v0b) at ( 1.5,  0.8);
\coordinate (v2b) at ( 2.0,  0);
\coordinate (v1b) at ( 1.5, -0.8);

\blackedge{v0a}{v1a}
\rededge{v0a}{v2a}
\blackedge{v1a}{v2a}

\rededged{v0b}{v1b}
\blackedge{v0b}{v2b}
\rededge{v1b}{v2b}

\rededged{v1a}{v0b}
\blackedged{v0a}{v1b}
\blackedge{v0a}{v2b}
\rededge{v1a}{v2b}
\blackedge{v2a}{v0b}
\blackedge{v2a}{v1b}

\foreach \p in {v0a,v1a,v2a,v0b,v1b,v2b}{
  \fill (\p) circle (2pt);
}
\node at (3.25, 0) {$\implies$};

\coordinate (u0a) at (-1.0+5.5,  0.8);
\coordinate (u2a) at (-0.5+5.5,  0.0);
\coordinate (u1a) at (-1.0+5.5, -0.8);
\coordinate (u0b) at ( 1.5+5.5,  0.8);
\coordinate (u2b) at ( 2.0+5.5,  0.0);
\coordinate (u1b) at ( 1.5+5.5, -0.8);

\blackedge{u0a}{u1a}
\rededge{u0a}{u2a}
\blackedge{u1a}{u2a}
\rededge{u0b}{u1b}
\blackedge{u0b}{u2b}
\rededge{u1b}{u2b}
\foreach \p in {u0a,u1a,u2a,u0b,u1b,u2b}{
  \fill (\p) circle (2pt);
}

\node at (5.975, 0) {$\sim_1$};

\end{tikzpicture}
    \caption{One-step cut-homotopy without monochromatic triangles.}
    \label{fig:cut-homotopy-graph-colorings}
\end{figure}

Observe that cut-homotopies of polymorphisms induce cut-homotopies of colorings: given a graph $G$, a coloring $\chi\colon\hat G\to\str A^r$ and 
a cut-homotopy $h\colon \str A^r\times \hat P_k\to\str A$, we can define the cut-homotopy of colorings
$$
\eta\colon\widehat{G\times P_k}\to\str A,\quad (vw,ij)\mapsto h(\chi(vw),ij).
$$
Moreover, both notions of cut-homotopies are transitive: if $f,g,h$ are polymorphisms and $f\sim_k g$ and $g\sim_\ell h$, then $f\sim_{k+\ell} g$, and similarly for colorings.

\section{Reductions}
\label{sec:reductions}

The goal of the current section is to show the following statement.

\thmReduction*

\begin{wrapfigure}{r}{0.35\textwidth}
    \centering
\begin{tikzpicture}[
  >={Stealth[length=4pt,width=3pt]},
  thick
]

\coordinate (C)     at (-0.2205,  0.0000);
\coordinate (Tapex) at ( 1.3938,  0.0000);
\coordinate (Tvtop) at ( 2.3031,  0.5250);
\coordinate (Tvbot) at ( 2.3031, -0.5250);

\draw
  (1.3055,  0.1530)
  arc (120:240:0.1767)
  -- (2.2148, -0.6780)
  arc (-120:0:0.1767)
  -- (2.4798,  0.5250)
  arc (0:120:0.1767)
  -- cycle;

\foreach \p in {Tapex, Tvtop, Tvbot} {
  \fill[white] (\p) circle (4.5pt);
  \draw[gray]  (\p) circle (4.5pt);
}

\fill[white] (C) circle (4.5pt);
\draw        (C) circle (4.5pt);

\tikzset{myarrow/.style={->, shorten >=9pt, shorten <=7pt}}

\draw[myarrow] (C)
  .. controls (0.4578, 0.9582) and (1.2990, 1.1332) ..
  node[above, midway, font=\small] {$\phi_1^S$}
  (Tvtop);

\draw[myarrow] (C) --
  node[above, midway, font=\small] {$\phi_2^S$}
  (Tapex);

\draw[myarrow] (C)
  .. controls (0.4578, -0.9582) and (1.2990, -1.1332) ..
  node[above, midway, font=\small] {$\phi_3^S$}
  (Tvbot);

\node[anchor=east,  font=\normalsize] at (-0.3950,  0.0) {$G_D$};
\node[anchor=west,  font=\normalsize] at ( 2.6700,  0.0) {$G_S$};

\coordinate (Dline) at (3.4, 0.0);

\coordinate (TvtopBot) at (2.3031,  0.3669);
\coordinate (TvbotTop) at (2.3031, -0.3669);

\draw[|-|, thin] (2.3031, 0.36) -- (2.3031, -0.36)
  node[midway, left, font=\small] {$k$};
\end{tikzpicture}
\end{wrapfigure}

We will prove this by showing that for an arbitrary relational structure $\str B$ we can find an equality-free gadget reduction from $\CSP(\str B)$ to $\col(\cat F)$. 
We point out that gadget reductions are a well-studied technique for $\CSP$s~\cite{HellNesetril, PolymorphismsAndHowToUseThem}, 
our contribution lies in using cut-homotopies to find the correct gadget. 
A \emph{gadget} consists of

\begin{itemize}
    \item a graph $G_D$ for the domain;
    \item a graph $G_S$ for every relation $S$ of $\str B$;
    \item an embedding $\phi^S_i\colon G_D\to G_S$ 
    for every $r$-ary relation $S$ of $\str B$ and $i\in r$.
\end{itemize}

The gadget has \emph{girth} $k$ if $k$ is the minimal distance (number of edges in the shortest path) between the images of any two distinct embeddings $\phi_i^S$.
We say that the gadget is \emph{equality-free} if it has girth at least one.

Given a gadget, the corresponding \emph{gadget replacement} takes a $\sigma$-structure $\str X$ and constructs a graph $G_\str X$ in the following way.
For every element $x\in \str X$ introduce a copy $G_x$ of $G_D$ and for every $\bar x\in S(\str X)$ introduce a copy $G_{\bar x}$ of $G_S$. Imagine each $x$ and $\bar x$ being ``replaced'' by the respective graphs. Afterwards, whenever we have $\bar x(i) = x$ for some coordinate $i$, identify $v\in G_x$ with 
$\phi_i^S(v)\in G_{\bar x}$, see Figure~\ref{fig:eq_free_gadget_replacement}. We denote the resulting graph by $G_\str X$.
\begin{figure}
    \centering
\newcommand{\arrowlabelA}{disassemble}
\newcommand{\arrowlabelB}{replace}
\newcommand{\arrowlabelC}{identify}

\begin{tikzpicture}[scale=0.58,
  >={Stealth[length=4pt,width=3pt]},
  thick
]

\def\mygeom{
  \coordinate (S0)   at (-1.0607,  0.0000);
  \coordinate (S1)   at ( 0.0000,  1.0607);
  \coordinate (S2)   at ( 1.0607,  0.0000);
  \coordinate (S3)   at ( 0.0000, -1.0607);
  \coordinate (OutT) at (-1.4489,  1.4489);
  \coordinate (OutR) at ( 1.4489,  1.4489);
  \coordinate (OutB) at ( 1.4489, -1.4489);
  \coordinate (OutL) at (-1.4489, -1.4489);
}

\def\mygeomBig{
  \coordinate (S0)   at (-1.5910,  0.0000);
  \coordinate (S1)   at ( 0.0000,  1.5910);
  \coordinate (S2)   at ( 1.5910,  0.0000);
  \coordinate (S3)   at ( 0.0000, -1.5910);
  \coordinate (OutT) at (-2.1733,  2.1733);
  \coordinate (OutR) at ( 2.1733,  2.1733);
  \coordinate (OutB) at ( 2.1733, -2.1733);
  \coordinate (OutL) at (-2.1733, -2.1733);
  \coordinate (hTop) at (-1.2548,  1.2548);
  \coordinate (hRt)  at ( 1.2548,  1.2548);
  \coordinate (hBot) at ( 1.2548, -1.2548);
  \coordinate (hLt)  at (-1.2548, -1.2548);
}

\def\sone{0}
\def\stwo{6.80}
\def\sthree{14.80}
\def\sfour{22.80}

\begin{scope}[shift={(\sone,0)}]
  \mygeom

  \coordinate (cT_S0Out_a) at (-1.0452,  0.5218);
  \coordinate (cT_S0Out_b) at (-1.1746,  1.0047);
  \coordinate (cT_OutS1_a) at (-1.0047,  1.1746);
  \coordinate (cT_OutS1_b) at (-0.5218,  1.0452);
  \coordinate (cT_S1S0_a)  at (-0.4596,  0.8132);
  \coordinate (cT_S1S0_b)  at (-0.8132,  0.4596);

  \coordinate (cR_S1Out_a) at ( 0.5218,  1.0452);
  \coordinate (cR_S1Out_b) at ( 1.0047,  1.1746);
  \coordinate (cR_OutS2_a) at ( 1.1746,  1.0047);
  \coordinate (cR_OutS2_b) at ( 1.0452,  0.5218);
  \coordinate (cR_S2S1_a)  at ( 0.8132,  0.4596);
  \coordinate (cR_S2S1_b)  at ( 0.4596,  0.8132);

  \coordinate (cB_S2Out_a) at ( 1.0452, -0.5218);
  \coordinate (cB_S2Out_b) at ( 1.1746, -1.0047);
  \coordinate (cB_OutS3_a) at ( 1.0047, -1.1746);
  \coordinate (cB_OutS3_b) at ( 0.5218, -1.0452);
  \coordinate (cB_S3S2_a)  at ( 0.4596, -0.8132);
  \coordinate (cB_S3S2_b)  at ( 0.8132, -0.4596);

  \coordinate (cL_S3Out_a) at (-0.5218, -1.0452);
  \coordinate (cL_S3Out_b) at (-1.0047, -1.1746);
  \coordinate (cL_OutS0_a) at (-1.1746, -1.0047);
  \coordinate (cL_OutS0_b) at (-1.0452, -0.5218);
  \coordinate (cL_S0S3_a)  at (-0.8132, -0.4596);
  \coordinate (cL_S0S3_b)  at (-0.4596, -0.8132);

  \fill[pattern=north east lines]
    (S0) .. controls (cT_S0Out_a) and (cT_S0Out_b) .. (OutT)
         .. controls (cT_OutS1_a) and (cT_OutS1_b) .. (S1)
         .. controls (cT_S1S0_a)  and (cT_S1S0_b)  .. cycle;
  \fill[pattern=north east lines]
    (S1) .. controls (cR_S1Out_a) and (cR_S1Out_b) .. (OutR)
         .. controls (cR_OutS2_a) and (cR_OutS2_b) .. (S2)
         .. controls (cR_S2S1_a)  and (cR_S2S1_b)  .. cycle;
  \fill[pattern=north east lines]
    (S2) .. controls (cB_S2Out_a) and (cB_S2Out_b) .. (OutB)
         .. controls (cB_OutS3_a) and (cB_OutS3_b) .. (S3)
         .. controls (cB_S3S2_a)  and (cB_S3S2_b)  .. cycle;
  \fill[pattern=north east lines]
    (S3) .. controls (cL_S3Out_a) and (cL_S3Out_b) .. (OutL)
         .. controls (cL_OutS0_a) and (cL_OutS0_b) .. (S0)
         .. controls (cL_S0S3_a)  and (cL_S0S3_b)  .. cycle;

  \draw (S0) .. controls (cT_S0Out_a) and (cT_S0Out_b) .. (OutT)
             .. controls (cT_OutS1_a) and (cT_OutS1_b) .. (S1)
             .. controls (cT_S1S0_a)  and (cT_S1S0_b)  .. cycle;
  \draw (S1) .. controls (cR_S1Out_a) and (cR_S1Out_b) .. (OutR)
             .. controls (cR_OutS2_a) and (cR_OutS2_b) .. (S2)
             .. controls (cR_S2S1_a)  and (cR_S2S1_b)  .. cycle;
  \draw (S2) .. controls (cB_S2Out_a) and (cB_S2Out_b) .. (OutB)
             .. controls (cB_OutS3_a) and (cB_OutS3_b) .. (S3)
             .. controls (cB_S3S2_a)  and (cB_S3S2_b)  .. cycle;
  \draw (S3) .. controls (cL_S3Out_a) and (cL_S3Out_b) .. (OutL)
             .. controls (cL_OutS0_a) and (cL_OutS0_b) .. (S0)
             .. controls (cL_S0S3_a)  and (cL_S0S3_b)  .. cycle;

  \foreach \p in {S0,S1,S2,S3,OutT,OutR,OutB,OutL} { \fill (\p) circle (3pt); }
\end{scope}

\draw[->, decorate, decoration={snake, amplitude=1.5pt, segment length=5pt, post length=4pt}]
  (2.50, 0.0) -- (4.30, 0.0);
\node[font=\normalsize] at (3.40, 0.5) {\arrowlabelA};

\begin{scope}[shift={(\stwo,0)}]
  \mygeomBig

  \coordinate (vS0)   at (S0);
  \coordinate (vS1)   at (S1);
  \coordinate (vS2)   at (S2);
  \coordinate (vS3)   at (S3);
  \coordinate (vOutT) at (OutT);
  \coordinate (vOutR) at (OutR);
  \coordinate (vOutB) at (OutB);
  \coordinate (vOutL) at (OutL);

  \coordinate (tTbl) at (-1.4117,  0.6692);
  \coordinate (tTbr) at (-0.6692,  1.4117);
  \coordinate (tTap) at (-1.6835,  1.6835);

  \coordinate (tRtl) at ( 0.6692,  1.4117);
  \coordinate (tRbl) at ( 1.4117,  0.6692);
  \coordinate (tRap) at ( 1.6835,  1.6835);

  \coordinate (tBtl) at ( 0.6692, -1.4117);
  \coordinate (tBtr) at ( 1.4117, -0.6692);
  \coordinate (tBap) at ( 1.6835, -1.6835);

  \coordinate (tLtr) at (-1.4117, -0.6692);
  \coordinate (tLbr) at (-0.6692, -1.4117);
  \coordinate (tLap) at (-1.6835, -1.6835);

  \coordinate (csT_blbr_a) at (-1.2384,  0.9909);
  \coordinate (csT_blbr_b) at (-0.9909,  1.2384);
  \coordinate (csT_brap_a) at (-1.0345,  1.4009);
  \coordinate (csT_brap_b) at (-1.3726,  1.4915);
  \coordinate (csT_apbl_a) at (-1.4915,  1.3726);
  \coordinate (csT_apbl_b) at (-1.4009,  1.0345);

  \coordinate (csR_blbr_a) at ( 0.9909,  1.2384);
  \coordinate (csR_blbr_b) at ( 1.2384,  0.9909);
  \coordinate (csR_brap_a) at ( 1.4009,  1.0345);
  \coordinate (csR_brap_b) at ( 1.4915,  1.3726);
  \coordinate (csR_apbl_a) at ( 1.3726,  1.4915);
  \coordinate (csR_apbl_b) at ( 1.0345,  1.4009);

  \coordinate (csB_blbr_a) at ( 0.9909, -1.2384);
  \coordinate (csB_blbr_b) at ( 1.2384, -0.9909);
  \coordinate (csB_brap_a) at ( 1.4009, -1.0345);
  \coordinate (csB_brap_b) at ( 1.4915, -1.3726);
  \coordinate (csB_apbl_a) at ( 1.3726, -1.4915);
  \coordinate (csB_apbl_b) at ( 1.0345, -1.4009);

  \coordinate (csL_blbr_a) at (-1.2384, -0.9909);
  \coordinate (csL_blbr_b) at (-0.9909, -1.2384);
  \coordinate (csL_brap_a) at (-1.0345, -1.4009);
  \coordinate (csL_brap_b) at (-1.3726, -1.4915);
  \coordinate (csL_apbl_a) at (-1.4915, -1.3726);
  \coordinate (csL_apbl_b) at (-1.4009, -1.0345);

  \fill[pattern=north east lines]
    (tTbl) .. controls (csT_blbr_a) and (csT_blbr_b) .. (tTbr)
           .. controls (csT_brap_a) and (csT_brap_b) .. (tTap)
           .. controls (csT_apbl_a) and (csT_apbl_b) .. cycle;
  \fill[pattern=north east lines]
    (tRtl) .. controls (csR_blbr_a) and (csR_blbr_b) .. (tRbl)
           .. controls (csR_brap_a) and (csR_brap_b) .. (tRap)
           .. controls (csR_apbl_a) and (csR_apbl_b) .. cycle;
  \fill[pattern=north east lines]
    (tBtl) .. controls (csB_blbr_a) and (csB_blbr_b) .. (tBtr)
           .. controls (csB_brap_a) and (csB_brap_b) .. (tBap)
           .. controls (csB_apbl_a) and (csB_apbl_b) .. cycle;
  \fill[pattern=north east lines]
    (tLtr) .. controls (csL_blbr_a) and (csL_blbr_b) .. (tLbr)
           .. controls (csL_brap_a) and (csL_brap_b) .. (tLap)
           .. controls (csL_apbl_a) and (csL_apbl_b) .. cycle;

  \draw (tTbl) .. controls (csT_blbr_a) and (csT_blbr_b) .. (tTbr)
               .. controls (csT_brap_a) and (csT_brap_b) .. (tTap)
               .. controls (csT_apbl_a) and (csT_apbl_b) .. cycle;
  \draw (tRtl) .. controls (csR_blbr_a) and (csR_blbr_b) .. (tRbl)
               .. controls (csR_brap_a) and (csR_brap_b) .. (tRap)
               .. controls (csR_apbl_a) and (csR_apbl_b) .. cycle;
  \draw (tBtl) .. controls (csB_blbr_a) and (csB_blbr_b) .. (tBtr)
               .. controls (csB_brap_a) and (csB_brap_b) .. (tBap)
               .. controls (csB_apbl_a) and (csB_apbl_b) .. cycle;
  \draw (tLtr) .. controls (csL_blbr_a) and (csL_blbr_b) .. (tLbr)
               .. controls (csL_brap_a) and (csL_brap_b) .. (tLap)
               .. controls (csL_apbl_a) and (csL_apbl_b) .. cycle;

  \foreach \p in {vS0,vS1,vS2,vS3,vOutT,vOutR,vOutB,vOutL,
                  tTbl,tTbr,tTap,
                  tRtl,tRbl,tRap,
                  tBtl,tBtr,tBap,
                  tLtr,tLbr,tLap} {
    \fill (\p) circle (3pt);
  }

  \tikzset{myarrow/.style={->, shorten >=5pt, shorten <=5pt}}

  \draw[myarrow] (vS0)   -- (tTbl);
  \draw[myarrow] (vS1)   -- (tTbr);
  \draw[myarrow] (vOutT) -- (tTap);

  \draw[myarrow] (vS1)   -- (tRtl);
  \draw[myarrow] (vS2)   -- (tRbl);
  \draw[myarrow] (vOutR) -- (tRap);

  \draw[myarrow] (vS2)   -- (tBtr);
  \draw[myarrow] (vS3)   -- (tBtl);
  \draw[myarrow] (vOutB) -- (tBap);

  \draw[myarrow] (vS3)   -- (tLbr);
  \draw[myarrow] (vS0)   -- (tLtr);
  \draw[myarrow] (vOutL) -- (tLap);

\end{scope}

\draw[->, decorate, decoration={snake, amplitude=1.5pt, segment length=5pt, post length=4pt}]
  (9.71, 0.0) -- (11.51, 0.0);
\node[font=\normalsize] at (10.61, 0.5) {\arrowlabelB};
 
\begin{scope}[shift={(\sthree,0)}]
  \mygeomBig

  \coordinate (vS0)   at (S0);
  \coordinate (vS1)   at (S1);
  \coordinate (vS2)   at (S2);
  \coordinate (vS3)   at (S3);
  \coordinate (vOutT) at (OutT);
  \coordinate (vOutR) at (OutR);
  \coordinate (vOutB) at (OutB);
  \coordinate (vOutL) at (OutL);

  \coordinate (tTbl) at (-1.4117,  0.6692);
  \coordinate (tTbr) at (-0.6692,  1.4117);
  \coordinate (tTap) at (-1.6835,  1.6835);

  \coordinate (tRtl) at ( 0.6692,  1.4117);
  \coordinate (tRbl) at ( 1.4117,  0.6692);
  \coordinate (tRap) at ( 1.6835,  1.6835);

  \coordinate (tBtl) at ( 0.6692, -1.4117);
  \coordinate (tBtr) at ( 1.4117, -0.6692);
  \coordinate (tBap) at ( 1.6835, -1.6835);

  \coordinate (tLtr) at (-1.4117, -0.6692);
  \coordinate (tLbr) at (-0.6692, -1.4117);
  \coordinate (tLap) at (-1.6835, -1.6835);

  \draw
    (-1.2867,0.5442)
    -- (-0.5442,1.2867)
    arc (-45.00:75.00:0.1767)
    -- (-1.6378,1.8542)
    arc (75.00:195.00:0.1767)
    -- (-1.5824,0.6235)
    arc (-165.00:-45.00:0.1767)
    -- cycle;

  \draw
    (0.5442,1.2867)
    -- (1.2867,0.5442)
    arc (-135.00:-15.00:0.1767)
    -- (1.8542,1.6378)
    arc (-15.00:105.00:0.1767)
    -- (0.6235,1.5824)
    arc (105.00:225.00:0.1767)
    -- cycle;

  \draw
    (0.6235,-1.5824)
    -- (1.6378,-1.8542)
    arc (-105.00:15.00:0.1767)
    -- (1.5824,-0.6235)
    arc (15.00:135.00:0.1767)
    -- (0.5442,-1.2867)
    arc (135.00:255.00:0.1767)
    -- cycle;

  \draw
    (-1.5824,-0.6235)
    -- (-1.8542,-1.6378)
    arc (165.00:285.00:0.1767)
    -- (-0.6235,-1.5824)
    arc (-75.00:45.00:0.1767)
    -- (-1.2867,-0.5442)
    arc (45.00:165.00:0.1767)
    -- cycle;

  \foreach \p in {tTbl,tTbr,tTap,
                  tRtl,tRbl,tRap,
                  tBtl,tBtr,tBap,
                  tLtr,tLbr,tLap} {
    \fill[white] (\p) circle (4.5pt);
    \draw[gray] (\p) circle (4.5pt);
  }

  \foreach \p in {vS0,vS1,vS2,vS3,vOutT,vOutR,vOutB,vOutL} {
    \fill[white] (\p) circle (4.5pt);
    \draw (\p) circle (4.5pt);
  }

  \tikzset{myarrow/.style={->, shorten >=5pt, shorten <=5pt}}

  \draw[myarrow] (vS0)   -- (tTbl);
  \draw[myarrow] (vS1)   -- (tTbr);
  \draw[myarrow] (vOutT) -- (tTap);

  \draw[myarrow] (vS1)   -- (tRtl);
  \draw[myarrow] (vS2)   -- (tRbl);
  \draw[myarrow] (vOutR) -- (tRap);

  \draw[myarrow] (vS2)   -- (tBtr);
  \draw[myarrow] (vS3)   -- (tBtl);
  \draw[myarrow] (vOutB) -- (tBap);

  \draw[myarrow] (vS3)   -- (tLbr);
  \draw[myarrow] (vS0)   -- (tLtr);
  \draw[myarrow] (vOutL) -- (tLap);

\end{scope}

\draw[->, decorate, decoration={snake, amplitude=1.5pt, segment length=5pt, post length=4pt}]
  (18.20, 0.0) -- (20.00, 0.0);
\node[font=\normalsize] at (19.10, 0.5) {\arrowlabelC};

\begin{scope}[shift={(\sfour,0)}]
  \mygeom

  \draw
    (-0.9357,-0.1250)
    -- ( 0.1250, 0.9357)
    arc (-45.00:75.00:0.1767)
    -- (-1.4032, 1.6196)
    arc (75.00:195.00:0.1767)
    -- (-1.2314,-0.0457)
    arc (-165.00:-45.00:0.1767)
    -- cycle;

  \draw
    (-0.1250, 0.9357)
    -- ( 0.9357,-0.1250)
    arc (-135.00:-15.00:0.1767)
    -- ( 1.6196, 1.4032)
    arc (-15.00:105.00:0.1767)
    -- (-0.0457, 1.2314)
    arc (105.00:225.00:0.1767)
    -- cycle;

  \draw
    ( 0.9357, 0.1250)
    -- (-0.1250,-0.9357)
    arc (135.00:255.00:0.1767)
    -- ( 1.4032,-1.6196)
    arc (-105.00:15.00:0.1767)
    -- ( 1.2314, 0.0457)
    arc (15.00:135.00:0.1767)
    -- cycle;

  \draw
    ( 0.1250,-0.9357)
    -- (-0.9357, 0.1250)
    arc (45.00:165.00:0.1767)
    -- (-1.6196,-1.4032)
    arc (165.00:285.00:0.1767)
    -- ( 0.0457,-1.2314)
    arc (-75.00:45.00:0.1767)
    -- cycle;

  \foreach \p in {S0,S1,S2,S3,OutT,OutR,OutB,OutL} {
    \fill[white] (\p) circle (4.5pt);
    \draw[gray]  (\p) circle (4.5pt);
  }

\end{scope}

\end{tikzpicture}
    \caption{If the girth is high, no new cliques are created.}
    \label{fig:eq_free_gadget_replacement}
\end{figure}
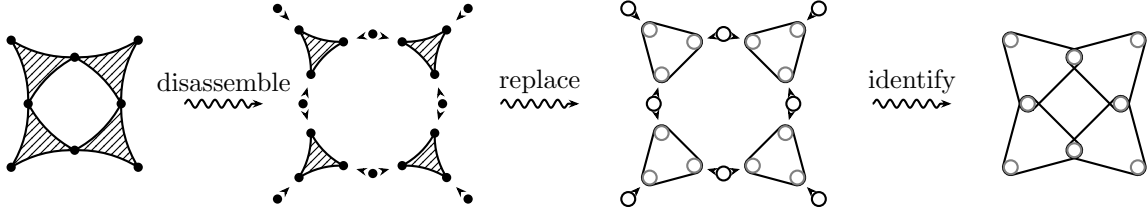
Observe that if the girth of the gadget is at least four, the image of every homomorphism from a clique to $G_\str X$ is fully contained in a copy $G_x$ of $G_D$ or in a copy $G_{\bar{x}}$ of $G_S$.
To find gadgets, we will employ the following two propositions that relate cut-homotopies of polymorphisms to cut-homotopies of colorings. They will be proved in Section~\ref{sec:ramseysToolbox}.

\begin{restatable}{proposition}{propGadgetOne}
\label{prop:homotopy_witness}
    For any two numbers $k$ and $n$ there is a graph $G$ and a coloring $\chi\colon\hat G\to\str A^n$ with the following property: any two polymorphisms $f,g\colon \str A^n\to \str A$ are $k$-step cut-homotopic whenever the colorings $f\circ \chi$ and $g\circ \chi$ are $k$-step cut-homotopic:
    $$ 
    f\circ \chi \sim_k g\circ\chi \Longrightarrow f\sim_k g
    $$
\end{restatable}

The converse implication holds for all $G$ and $\chi$.

\begin{restatable}{proposition}{lemRamsey}
\label{prop:homotopy_witness2}
    Fix a number $n$, a finite number of graphs
    $G_i$ and colorings $\chi_i\colon \hat G_i\to \str A^n$, for $i \in r$. There exists a graph $H$, a coloring $\chi\colon \hat H\to \str A^n$ and embeddings $\phi_i\colon G_i\to H$ with $\chi\circ\hat\phi_i = \chi_i$ for all $i \in r$, such that for every coloring $\eta\colon\hat H\to \str A$, there is a polymorphism $f\colon \str A^n\to \str A$ and $2$-step cut-homotopies
    $$
    f\circ \chi_i \sim_2 \eta\circ\hat\phi_i
    $$
    for all $i \in r$.
    Moreover, the distance between the images of any two distinct $\phi_i$ is at least four.
\end{restatable}

\begin{proof}[Proof of Theorem~\ref{thm:reduction}]
Take a minion homomorphism $\alpha\colon \pol(\str A)\to \id$ that collapses cut-homotopies. For any finite relational structure $\str B$, we will find a gadget replacement that reduces $\CSP(\str B)$ to $\col(\cat F)$.
Let $G_D$ be the graph obtained from Proposition~\ref{prop:homotopy_witness}, for $k=4$ and $n=B$. Denote the obtained coloring by $\chi_D\colon \hat G_D\to\str A^B$.

Given a relation $S \subseteq B^r$ of $\mathbb B$ with arity $r$ we define a graph $G_S$ in the following way. For every coordinate $i \in r$, let $\chi_i\colon \hat G_D\to\str A^S$ be the coloring defined by
$$
\chi_i = \pi_i^*\circ \chi_D
$$
where $\pi_i\colon S\to B$ is the $i$-th projection and $\pi_i^*$ denotes the precomposition map $\str A^B\to\str A^S$, $a\mapsto a\circ\pi_i$.
Applying Proposition~\ref{prop:homotopy_witness2} with $n=S$ to these colorings, we obtain the graph $G_S$, embeddings $\phi_i^S\colon G_D\to G_S$ for all $i \in r$ and a coloring $\chi_S\colon \hat G_S\to \str A^S$. By the proposition, the girth of this gadget is at least four, we show that it yields a reduction.

First we show completeness, that is, whenever $\str X$ is a YES-instance of $\CSP(\str B)$ then $G_\str X$ is a YES-instance of $\col(\cat F)$. Indeed, given a homomorphism $f\colon \str X\to \str B$, we  construct a coloring $\eta\colon\hat G_\str X\to \str A$ as follows. 
Whenever an edge $vw\in\hat G_\str X$ is contained in a copy $G_x$ of $G_D$, we take the function $\chi_D(vw)\in \str A^B$ and evaluate at $f(x)$ to obtain the color $vw$. Similarly, if $vw$ is contained in a copy $G_{\bar x}$ of some $G_S$, we take $\chi_S(vw)$ and evaluate at $(f\circ\bar x)$. 
$$
\eta(vw) = 
\begin{cases}
\chi_D(vw)(f(x)) &\text{if } v,w \in G_x  \\
\chi_S(vw)(f\circ \bar x) &\text{if } v,w\in G_{\bar x}
\end{cases}
$$
Note that every edge $vw$ is contained in some $G_x$ or $G_{\bar x}$ and the condition $\pi_i^*\circ\chi_D=\chi_i = \chi_S\circ\hat\phi^S_i$ ensures that $\eta$ is well-defined. Moreover, since the girth of the gadget is at least four, every clique in $G_\str X$ is fully contained in one of the $G_x$ or $G_{\bar x}$, hence the coloring $\eta$ is $\cat F$-free if and only if $\chi_D, \chi_S$ are, and they are $\cat F$-free at each coordinate, by definition.
Next we show soundness, that is, whenever $ G_\str X$ is a YES-instance of $\col(\cat F)$ then $\str X$ is a YES-instance of $\CSP(\str B)$.
Indeed, given a coloring $\eta\colon\hat G_\str X\to \str A$, we show the existence of a homomorphism $g\colon \str X\to \str B$ as follows.
Let $g\colon X\to B$ be any map that satisfies $g(x)=\alpha(f)$, whenever there is a polymorphism $f: \str A^B \rightarrow \str A$ such that $\eta_x\sim_2 f\circ\chi_D$, where $\eta_x$ indicates the restriction of $\eta$ to $G_x$. Such a map exists, since 
whenever two polymorphisms $f$ and $f'$ satisfy the above condition, we get $f'\circ \chi_D \sim_4 f\circ\chi_D$, which, by the definition of $G_D$, implies $f\sim_4 f'$ and hence $\alpha(f)=\alpha(f')$, by assumption.
We prove that $g$ describes a homomorphism. For this purpose we take  $\bar x\in S(\str X)$ and we show that $g\circ\bar x \in S$.
Let $\eta_{\bar x}$ be the restriction of $\eta$ to $G_{\bar x}$.
By the definition of $G_S$, there is a polymorphism $f\colon\str A^S\to\str A$ and 
2-step cut-homotopies 
$$
 f\circ\chi_i\sim_2\eta_{\bar x}\circ\hat\phi_i
$$
for all $i \in r$. Moreover, by definition, $f\circ \chi_i =f\circ\pi_i^*\circ \chi_D$ and $\eta_{\bar x}\circ\hat\phi^S_i = \eta_{\bar x(i)}$, so the function $i\mapsto 
g(\bar x(i))$ indeed belongs to $S$ since $\alpha(f)\in S$ and 
$$
g(\bar x(i)) = 
\alpha(f\circ\pi_i^*)= \pi_i(\alpha(f)).
$$
The first equality holds by the definition of $g$, and the second one follows from the fact that $\alpha$ is a minion homomorphism.
\end{proof}

As usual with gadget reductions, the minion homomorphism $\alpha$ was used only to guarantee soundness. For completeness, we used the gadget's girth and the existence of certain colorings.
We remark that, using essentially the same proof, one can show a slight refinement of Theorem~\ref{thm:reduction}: If there is a minion homomorphism $\pol\str A\to\pol\str B$ that collapses cut-homotopies, then there is an efficient reduction from $\CSP(\str B)$ to $\col(\cat F)$.

\section{There are no weird homotopies!}
\label{sec:therearenoweirdhomotopies}
This section is devoted to the proof of the following theorem.

\begin{restatable}{theorem}{noHomotopies}
\label{thm:noHomotopies}
    If there is a minion homomorphism $\pol(\str A) \to \id$, then there is a minion homomorphism $\pol(\str A)\to \id$ that collapses cut-homotopies.
\end{restatable}

\subsection{Preliminaries}
We review standard notions from the algebraic approach to CSPs that are needed below, and which can be found in e.g.,~\cite{PolymorphismsAndHowToUseThem}.
An \emph{endomorphism} of $\str A$ is a homomorphism from $\str A$ to itself.
A structure $\str A$ is called a \emph{core} if every endomorphism of $\str A$ is an isomorphism.
For every structure $\str A$ there is a core $\str A'$ that is \emph{homomorphically equivalent} to $\str A$, i.e., there are homomorphisms $\str A \to \str A'$ and $\str A' \to \str A$.
The structure $\str A'$ is unique up to isomorphism, justifying the name \emph{the core of} $\str A$, and embeds into $\str A$.
Let $\str A'$ be the core of $\str A$ and $\iota \colon \str A' \to \str A$, $\psi\colon \str A \to \str A'$ be homomorphisms. These maps induce a minion homomorphism $\pol(\str A) \to \pol(\str A'), f \mapsto \psi \circ f\circ \iota_*$, where $\iota_*$ denotes the map $(\str A')^r \to \str A^r, a \mapsto \iota\circ a$ for all $r$. We refer to it as \emph{the minion homomorphism associated with $\iota$ and $\psi$.}

The \emph{unary minor} of a function $f\colon \str A^n \to \str A$ is the unary map $e_f(x) \coloneq f(x,\ldots,x)$.
We say that $f$ is \emph{idempotent} if $e_f= \id_\str A$.
We denote the minion consisting of all idempotent polymorphisms of $\str A$ by $\pol(\str A)_\idem$.
If $\str A$ is a core, there is a minion homomorphism $\alpha_\idem\colon\pol(\str A) \to \pol(\str A)_\idem, f\mapsto e_f^{-1}\circ f$.
We say that a cut-homotopy $h\colon \str A^n \times \hat P_k \to \str A$ is \emph{idempotent} if $h(a,\dots,a,ij) = a$, for all $a \in \str A, ij\in \hat P_k$.

Let $S\subseteq A$ be a subset preserved by $\pol(\str A)_\idem$ and $\theta \subseteq S^2$ be an equivalence relation preserved by $\pol(\str A)_\idem$. For every number $n$, $\theta$ induces an equivalence relation on $S^n$  by requiring elements to be $\theta$-related componentwise; we also refer to this equivalence relation by $\theta$.
Saying that $S$ and $\theta$ are \emph{preserved} by $\pol(\str A)_\idem$ precisely means that for every number $n$, every $n$-ary $f \in \pol(\str A)$ induces a function
\begin{equation*}
    S^n/_\theta \to S/_\theta, \quad a/_\theta \mapsto f(a)/_\theta,
\end{equation*} 
the \emph{action of $f$ on $S/_\theta$}.
The pair $(S, \theta)$ is called a \emph{naked set} of $\pol(\str A)_\idem$ if $\theta$ has two equivalence classes and every function contained in $\pol(\str A)_\idem$ acts as a projection on $S/_\theta$. The assignment $\pol(\str A)_\idem \to \id$ sending $f\colon \str A^n \to \str A$ to $i \in \{1,\dots, n \}$, where $f$ acts as the $i$-th coordinate projection on $S/_\theta$ is easily seen to be a minion homomorphism.
We refer to it as \emph{the minion homomorphism induced by $(S,\theta)$.} 
Finally, we say that a minion homomorphism $\alpha: \pol (\str A) \rightarrow \pol (\str B)$ \emph{preserves cut-homotopies} if whenever $f, g \in \pol (\str A)$ are cut-homotopic, then so are $\alpha(f), \alpha(g)$.

\subsection{The master proof of Theorem~\ref{thm:noHomotopies}}

Let $\str A'$ be the core of $\str A$. 
The existence of a minion homomorphism $\pol(\str A) \to \id$ implies the existence of a minion homomorphism $\pol(\str A')_\idem \to \id$, see e.g.,~\cite{Wonderland}. 
By~\cite[Prop.~4.14]{NakedSet} the latter statement implies the existence of a naked set $(S,\theta)$ of $\pol(\str A')_\idem$. The induced minion homomorphism $\alpha \colon \pol(\str A')_\idem \to \id$ collapses idempotent cut-homotopies by Proposition~\ref{prop:noHomotopies}. Further, the combination of Lemma~\ref{lem:str_to_core} and Lemma~\ref{lem:core_to_idemp} yields a minion homomorphism $\beta \colon \pol(\str A) \to \pol(\str A')_\idem$ that
preserves cut-homotopies. Moreover, Lemma~\ref{lem:core_to_idemp} shows that $\beta(f)\sim_1 \beta(g)$ is witnessed by an idempotent cut-homotopy, for all $f\sim_1 g$ in $\pol(\str A)$.
Thus, the composition $\alpha \circ \beta \colon \pol(\str A) \to \id$ yields the desired minion homomorphism that collapses cut-homotopies.

$$
\pol(\str A) 
    \xrightarrow[\text{Lemma~\ref{lem:str_to_core}}]{\substack{\text{preserves}\\\text{cut-homotopies}}}
\pol(\str A') 
    \xrightarrow[\text{Lemma~\ref{lem:core_to_idemp}}]{\substack{\text{makes cut-homotopies}\\\text{idempotent}}}
\pol(\str A')_{\idem} 
    \xrightarrow[\text{Proposition~\ref{prop:noHomotopies}}]{\substack{\text{collapses idempotent}\\\text{cut-homotopies}}}
\id
$$

\subsection{Auxiliary results for the proof of Theorem~\ref{thm:noHomotopies}}

Since $\str A$ corresponds to some finite family $\cat F$ of forbidden edge-colored cliques, also the core $\str A'$ of $\str A$ corresponds to some finite family $\cat F'$ of edge-colored cliques. Namely, since $\str A'$ is an induced substructure of $\str A$ we may choose $\cat F'$ so that the $\cat F'$-free colorings of cliques are precisely the $\cat F$-free colorings with colors contained in the domain of $\str A'$. 
For this reason we may (and will) verify that a map $(\str A')^n \to \str A'$ or $(\str A')^n \times \hat P_1 \to \str A'$ is a homomorphism in the same way we do for $\str A$.

\begin{lemma}\label{lem:str_to_core}
        Let $\str A'$ be the core of $\str A$ and let $\iota\colon \str A'\to \str A$, $\psi\colon \str A\to \str A'$ be homomorphisms. Then the
        minion homomorphism $\alpha\colon\pol(\str A)\to\pol(\str A')$ associated with $\iota$ and $\psi$ preserves cut-homotopies.
\end{lemma}

\begin{proof}
     Let $h\colon \str A^n \times \hat P_1 \to \str A$ be a cut-homotopy between polymorphisms $h_0$ and $h_1$ of $\str A$. The map 
     \begin{equation*}
         \psi\circ h\circ (\iota_* \times \id_{\hat P_1}) \colon (\str A')^n \times \hat P_1 \to \str A^n \times \hat P_1 \to \str A \to \str A'
     \end{equation*}
     is a cut-homotopy between $\psi \circ h_0 \circ \iota_* = \alpha(h_0)$ and $\psi \circ h_1 \circ \iota_* = \alpha(h_1)$.
\end{proof}

\begin{lemma}\label{lem:core_to_idemp}
        Let $\str A$ be a core. 
        Then the minion homomorphism $\alpha_\idem\colon \pol(\str A)\to \pol(\str A)_\idem$
        preserves cut-homotopies and 
        whenever $f,g\in\pol(\str A)$ are cut-homotopic, the existence of a cut-homotopy $\alpha_\idem(f)\sim_1\alpha_\idem(g)$ is witnessed by an idempotent cut-homotopy.
\end{lemma}

\begin{proof}
    Let $h$ be a cut-homotopy between two polymorphisms $h_{0}$ and $h_{1}$ of $\str A$.
    We will construct an idempotent cut-homotopy $h''$ such that $h''_{0} = \alpha_{\idem}(h_{0})$ and $h''_{1} = \alpha_{\idem}(h_{1})$.
    Let $e\colon \str A\times {\hat P_1} \to \str A$ be the unary minor of $h$, i.e., $e(x,ij)\coloneq h(x,\ldots,x, ij)$.
    Clearly, $e$ is a homomorphism.
    
    Note that $e_{0},e_{1}$ are endomorphisms of $\str A$ and, since $\str A$ is a core, they have
    inverses.
    Define $h'\coloneq e_0^{-1}\circ h$
    and let $e'$ be its unary minor, so $e' = e_0^{-1}\circ e$. 
    We now verify that $e'_{01}$ is invertible. Since $\str A$ is a core it suffices to show that $(e_{01}')^2$ is an endomorphism of $\str A$. We do that by checking that for every $K_m$ and every coloring $\chi \colon \hat K_m \to \str A$, also $(e'_{01})^2\circ \chi$ is a coloring of $K_m$.
    For every vertex $v\in K_m$, define the cut $c_{v}\coloneq K_m\setminus \{v\} \mid v$.
    Since $e'$ is a homomorphism with $e'_{0} = \id_{\str A}$, the coloring $e'\circ (\chi, \hat c_v)\colon \hat K_m \to \str A$ is valid and it differs from $\chi$ only on edges incident to $v$, where it agrees with $e'_{01}\circ \chi$.
    Repeating this for every $v\in K_m$ gives a valid coloring, which coincides with $(e_{01}')^2\circ \chi$ at every edge. 
    Since $\str A$ is finite and $e'_{01}$  bijective, there exists some number $N$ such that $(e'_{01})^N = (e'_{01})^{-1}$. We define a map $\tilde e \colon \str A \times \hat P_1 \to \str A$ as follows
    \begin{align*}
        \tilde e_0 &= \id_{\str A},\\
        \tilde e_{01} &= (e'_{01})^N,\\
        \tilde e_1 &= (e'_{01})^{2N};
    \end{align*}
    and verify that it is a homotopy. To this end we show that for every $K_m$, every coloring $\chi\colon \hat K_m \to \str A$ and every cut $c\colon K_m \to P_1$, the map $\chi' \coloneq \tilde e \circ (\chi, \hat c)$ is a valid coloring. 
    Let $v \in K_m$ be a vertex with $c(v) = 1$, let $c_v \coloneq K_m \setminus \{v\} \mid v$ and set $\chi_1 \coloneq e' \circ (\chi,\hat c_v)$. Recursively we define $\chi_{i+1} \coloneq e' \circ (\chi_i,\hat c_v)$ for $i < N$ to obtain a valid coloring $\chi_N \colon \hat K_m \to \str A$ that agrees with $\chi'$ on all edges of the form $vw$ and $w w'$ where $w, w' \in K_m$ with  $c(w)= c(w')=0$. Here we crucially use that $e'_0 = \id_{\str A}$.
    Replacing $\chi$ by $\chi_N$ we repeat this procedure for every further vertex $v' \in K_m$ with $c(v') = 1$ to obtain a coloring that agrees with $\chi'$ proving it to be valid.
    
    We may now define the homotopy $h''\colon \str A^n \times \hat P_1 \to \str A$ as the composition
    \begin{equation*}
        \str A^n \times \hat P_1 \xrightarrow{(h'\times\id_{\hat P_1})} \str A \times \hat P_1 \xrightarrow{\tilde e}  \str A.
    \end{equation*}
    Let $e''$ be its unary minor and observe that $e''_{0} = \id_{\str A} = e''_{01}$.
    Furthermore, $h''_1 = \tilde e_1 \circ e_0^{-1} \circ h_1$, so it will follow that $\alpha_{\idem}(h_1)  = h''_1$ once we have established that $e_1'' = \id_{\str A}$.
    Assuming the contrary we obtain distinct $a,b\in \str A$ with $e_{1}''(a) = b$ and argue that the map $\phi\colon \str A \to \str A$ with $\phi(a)=b$ and $\phi(a')=a'$ for every $a'\in \str A\setminus \{a\}$ is a homomorphism. This contradicts $\str A$ being a core.
    To do so we show that for every clique $K_m$ and every coloring $\chi\colon \hat K_m\to \str A$, $\phi \circ \chi$ is a valid coloring of $K_m$ itself.
    For every edge $vw \in \hat K_m$ with $\chi(vw)=a$, we consider the cut $c_{vw}\colon K_m\setminus \{v,w\}\mid \{v,w\}$, and note that $e''\circ (\chi, \hat c_{vw})$ is a valid coloring which only differs from $\chi$ on the edge $vw$ where it takes the value $b$.
    Repeating this for every further $a$-colored edge results in a valid coloring that coincides with $\phi\circ \chi$, so $\phi$ is a homomorphism. 
\end{proof}

For the next proposition, we need the following Ramsey-theoretic lemma, which we prove in Section~\ref{sec:ramseysToolbox}.

\begin{restatable}{lemma}{mapFactory}
\label{lem:mapFactory}
    Let $n$ be a number, and for each $a\in \str A^n$ pick a subset $U_{a}\subseteq A$. Assume that for every graph $G$ and every coloring $\chi\colon \hat G\to \str A^n$, there is a coloring $\eta\colon \hat G\to \str A$ such that 
    \begin{equation*}
        \forall vw\in\hat G: \eta(vw)\in U_{\chi(vw)}.
    \end{equation*}
    Then there exists a polymorphism 
    $f \colon \str A^n \to \str A$ 
    with $f(a)\in U_{a}$ for all $a \in \str A^n$.
\end{restatable}

We are now prepared to establish that there are no ``weird homotopies'', i.e., an idempotent cut-homotopy between two polymorphisms that act as different projections on a given naked set $(S,\theta)$.

\begin{proposition}
\label{prop:noHomotopies}
    Let $\str A$ be a core and $(S,\theta)$ be a naked set of $\pol(\str A)_\idem$. Then the minion homomorphism $\alpha\colon \pol(\str A)_\idem\to\id$ induced by $(S, \theta)$ collapses idempotent cut-homotopies, i.e., if $f,g\in\pol(\str A)_\idem$ and $f\sim_1 g$ with $h$ idempotent, then $\alpha(f)=\alpha(g)$.
\end{proposition}

\begin{proof}

Striving for a contradiction we assume that $h$ is an idempotent cut-homotopy between $n$-ary polymorphisms $h_0,h_1$ that act as different projections on $S/_\theta$, without loss of generality, the first and second projection. Then, $h' \colon \str A^2 \times \hat P_1 \to \str A, (x,y,ij) \mapsto h_{ij}(x,y,\dots, y)$ is a cut-homotopy between binary polymorphisms of $\str A$ that act as different projections on $S/_\theta$. Hence, we may assume $h$ to be binary, and that $h_0, h_1$ act on $S/_\theta$ as the first and second projection, respectively.

Let $d$ be the homomorphism $(h_0, h_1) \colon \str A^2\to \str A^2, (a,b) \mapsto (h_0(a,b),h_1(a,b))$, and choose $N$ so that $d^N$ is a retraction, i.e., $d^N = d^{2N}$. Let $D$ be the image of $d^N$, and note that $d^N$ restricts to the identity on $D$.

\begin{claim}
The composition $f\coloneq h_{01}\circ d^N$ is an idempotent polymorphism of $\str A$ satisfying $f|_D = h_{01}|_D$.
\end{claim}
\begin{proof}[Proof of the claim]
We show that for every $K_m$ and every coloring $\chi\colon \hat K_m\to \str A^2$, 
the map $f \circ \chi$ is a homomorphism $\hat K_m \to \str A$.
Define $\chi_0\coloneq d^N\circ \chi$.

Fix a vertex $v\in K_m$ and consider the two cuts $c_{0} \coloneq K_m\setminus\{v\}\mid v$ and
$c_1\coloneq  v\mid K_m\setminus\{v\}$.
Applying $h$ to $\chi_0$ along each cut yields a coloring $\hat K_m \to \str A$ and we let $\chi_1\colon \hat K_m \to \str A^2$ denote their product. More formally, $\chi_1$ is defined to be the homomorphism $\big( h \circ (\chi_0,\hat c_0), h \circ (\chi_0,\hat c_1) \big)\colon \hat K_m \to \str A^2$.
For every $w w' \in \hat K_m$ it holds
    \begin{equation*}
            \chi_1(w w') =
    \begin{cases}
    \big(h_0\circ\chi_0(w w'), h_1\circ \chi_0(w w') \big) = d\circ \chi_0(w w') &\text{if } v\notin \{w, w'\} \\ 
    \big(h_{01}\circ\chi_0(w w'), h_{01}\circ\chi_0(w w')\big) &\text{if } v\in \{w, w'\}.
    \end{cases}
    \end{equation*}
    Observe that in the latter case both entries of $\chi_1(w w')$ agree. We iterate this construction, obtaining $\chi_{i+1}$ from $\chi_i$ in the same way $\chi_1$ is obtained from $\chi_0$ until we arrive at the coloring $\chi_N$. The aforementioned observation in combination with the fact that $h$ is idempotent guarantees that for every $w w' \in \hat K_m$ it holds 
    \begin{equation*}
            \chi_N(w w') =
    \begin{cases}
    d^N \circ \chi_0(w w') = \chi_0(w w')&\text{if } v\notin \{w, w'\}\\ 
    \big(h_{01}\circ\chi_0(w w'),h_{01}\circ\chi_0(w w')\big) &\text{if } v\in \{w, w'\}.
    \end{cases} 
    \end{equation*}
    Replacing $\chi_0$ with $\chi_N$ we repeat this procedure for every vertex of 
    $K_m$.
    Once we have exhausted all vertices, we call the resulting coloring $\chi'\colon \hat K_m \to \str A^2$. For every edge $w w'$ of $K_m$, as it is incident to precisely two vertices of $K_m$, we have
    \begin{align*}
        \chi'(w w') &= \Big( h_{01}\big(h_{01}\circ \chi_0(w w'), h_{01}\circ \chi_0(w w')\big), h_{01}\big(h_{01}\circ \chi_0(w w'), h_{01}\circ \chi_0(w w')\big) \Big)\\
        &= \big( h_{01}\circ \chi_0(w w'), h_{01}\circ \chi_0(w w') \big)\\
        &= \big( f \circ \chi(w w'), f \circ \chi(w w') \big),
    \end{align*}
    in particular, $f \circ \chi$ is a homomorphism $\hat K_m \to \str A$.
    Thus, $f$ is a homomorphism; its idempotence is clear.
\end{proof}

Without loss of generality we may assume that $f$ acts as the first projection on $S/_\theta$. Thus, by definition of $f$ each pair $(a,b) \in S^2 \cap D$ satisfies $h_{01}(a,b)/_\theta = a/_\theta$.
Further, recall that $h_0$ and $h_1$ act as the first and second projection on $S/_\theta$, respectively.
 
We now fix $a,b\in S$ in different classes of $\theta$.
We use Lemma~\ref{lem:mapFactory} to show that there is a polymorphism $g: \str A^2 \rightarrow \str A$ with $g(a,b)/_\theta = b/_\theta$ and $g(a',b')/_\theta = a'/_\theta$ whenever $a',b'\in S$ with $(a,b)/_\theta\not= (a',b')/_\theta$, contradicting the fact that $(S,\theta)$ is a naked set.
To this end, we verify that for every graph $G$ and every coloring $\chi\colon \hat G\to \str A^2$, there is a coloring $\eta\colon \hat G\to \str A$ such that for all $vw \in \hat G$ it holds:
\begin{equation*}
    \eta(vw) \in 
\begin{cases}
    b/_\theta, &\text{ if } \chi(v w) \in S^2 \text{ and } \chi(v w)/_\theta = (a,b)/_\theta \\
    a'/_\theta, &\text{ if } \chi(vw) = (a',b') \in S^2\text{ and } (a',b')/_\theta \not= (a,b)/_\theta.
\end{cases}
\end{equation*}

So let $G$ be a graph and $\chi = (\chi_1,\chi_2) \colon \hat G\to \str A^2$ a coloring. Note that the coloring $d^N\circ \chi$ agrees with $\chi$ modulo
$(S,\theta)$ on all $vw \in \hat G$ with $\chi(vw)\in S^2$ and additionally its image is contained in $D$. 
Take any edge $v_0 w_0 \in \hat G$ with $\chi(v_0 w_0)/_\theta = (a,b)/_\theta$ and consider the cut $c \coloneq G\setminus \{v_0,w_0\} \mid \{v_0,w_0\}$. 
Define $\eta'\colon \hat G\to \str A$ to be the coloring $h \circ (d^N\circ \chi,\hat c)$. Using the properties of $h$, in particular that $h_{01}|_D$ acts as the first projection on $S/_\theta$ we obtain that for $v w \in \hat G$ it holds
\begin{equation*}
    \eta'(v w)  \in
\begin{cases}
    h_{1}(d^N \circ \chi(v w))/_\theta = b/_\theta &\text{ if } v w = v_0 w_0 \\
    h_{01}(d^N \circ \chi(v w))/_\theta =  a'/_\theta &\text{ if } \chi(v w) = (a',b') \in S^2\text{ and } |\{v, w\}\cap \{v_0,w_0\}| = 1 \\
    h_{0}(d^N \circ \chi(v w))/_\theta = a'/_\theta &\text{ if } \chi(v w) = (a',b') \in S^2 \text{ and } |\{v, w\} \cap \{v_0,w_0\}| = 0.
\end{cases}
\end{equation*}
Define the coloring $\chi' \coloneq (\eta',\chi_2)\colon \hat G \to \str A$ and observe that on all $vw \in \hat G\setminus \{v_0 w_0\}$ with $\chi(vw)\in S^2$, $\chi(vw)$ and $\chi'(vw)$ agree modulo $\theta$ and furthermore $\chi'(v_0 w_0)/_\theta = (b,b)/_\theta$. 
We repeat the same procedure with $\chi'$ in place of $\chi$ with every further edge $v_0' w_0' \in \hat G$ with $\chi(v_0' w_0')/_\theta = (a,b)/_\theta$ until we have exhausted all such edges. We obtain a coloring $\tilde \eta \colon \hat G \to \str A$ such that for all $vw \in \hat G$ it holds
\begin{equation*}
        \tilde \eta(v w) \in 
\begin{cases}
    b/_\theta &\text{ if } \chi(v w)/_\theta = (a,b)/_\theta\\
    a'/_\theta &\text{ if } \chi(v w) = (a',b') \in S^2 \text{ and } (a',b')/_\theta \not= (a, b)/_\theta.
\end{cases}
\end{equation*}
Thus, $\tilde \eta$ is the desired coloring $\eta$.
\end{proof}

\section{Ramsey's toolbox}
\label{sec:ramseysToolbox}

In this section we prove the Ramsey-theoretical claims used in Section~\ref{sec:reductions} and Section~\ref{sec:therearenoweirdhomotopies}, namely Propositions~\ref{prop:homotopy_witness} and~\ref{prop:homotopy_witness2} and Lemma~\ref{lem:mapFactory}.

\begin{definition}
    A category
    $\cat C$ is said to be \emph{Ramsey} if for any two objects $\mathbf{A}$ and $\mathbf{B}$ and any finite set $N$ there exists an object $\mathbf{C}$ such that for all maps $\hom(\mathbf{A},\mathbf{C})\xrightarrow{\chi} N$ there exists a morphism $\mathbf{B}\xrightarrow{g} \mathbf{C}$ such that the composition
    $$
   \chi\circ g_*:\hom(A,B)\xrightarrow{g_*} \hom(A,C)\xrightarrow{\chi} N
    $$
    is a constant map, where $g_*$ is the postcomposition map $f\mapsto g\circ f$. Any such object $\mathbf{C}$ is called a \emph{Ramsey witness} of $\mathbf{A},\mathbf{B}$ and $N$.
\end{definition}

Given a number $n$ we define a category $\cat C$, where the objects are $\str A^n$-colored,
linearly ordered graphs, that is, triples $\mathbf{G}=(G,\chi,<)$ where $G$ is a graph, $\chi\colon \hat G\to \str A^n$ is a homomorphism and $<$ is a linear order on the vertices of $G$. A morphism in this category $\iota\colon\mathbf{G}=(G,\chi,<)\to\mathbf{G'}=(G',\chi',<')$ is an embedding between the graphs $G$ and $G'$, that is monotone with respect to the linear orders and that respects the colorings, meaning $\chi' \circ \hat\iota = \chi$. 
Since the patterns we are forbidding are cliques, a theorem by Ne\v set\v ril and Rödl~\cite{FullNesetrilRodl1983} implies that for any given $n$, $\cat C$ enjoys the \emph{Ramsey property}. For a more recent treatment see Theorem~2.9 in \cite{Hubi_ka_2026}.

\begin{theorem}[Special case of~\cite{FullNesetrilRodl1983}]
For every $n$ the category $\cat C$
is Ramsey.
\end{theorem}
Let $\cat D \colon \cat C\op \to \set$ be a contravariant functor from $\cat C$ to the category of sets. We call a family $\{s_\mathbf{G}\}_{\mathbf{G}\in \cat C}$  a \emph{solution of} $\cat D$ if $s_\mathbf{G}\in \cat D(\mathbf{G})$ for each $\mathbf{G}\in \cat C$, and for every embedding $\iota\colon \mathbf{G} \to \mathbf{H}$ in $\cat C$, we have $\cat D_\iota(s_\mathbf{H}) = s_\mathbf{G}$. 
We say that a category is \emph{confluent} if for any three objects and two morphisms 
$$
\mathbf{B} \leftarrow\mathbf{A} \rightarrow\mathbf{C}
$$ 
there is an object $\mathbf{D}$ and two morphisms 
$$
\mathbf{B} \rightarrow \mathbf{D} \leftarrow \mathbf{C}.
$$
The category $\cat C$ described above is confluent: indeed given two colored ordered graphs $\mathbf{B}$ and $\mathbf{C}$ the corresponding $\mathbf{D}$ can be obtained by taking the disjoint union and completing the resulting partial order in an arbitrary way to a linear order.

\begin{theorem}[Theorem~1.2 in \cite{Hadek_konig:ramsey}]
\label{thm:konigramsey}
    Let $\cat C$ be a category that is confluent and Ramsey and let $\cat D\colon \cat C\op\to\set$ be a functor, such that the set $\cat D(\mathbf C)$ is finite and nonempty for each $\mathbf C\in \cat C$. Then there is a solution of $\cat D$. 
\end{theorem}

We have gathered all necessary ingredients to prove the following lemma, from which Proposition~\ref{prop:homotopy_witness2} and Lemma~\ref{lem:mapFactory} will follow.
\begin{lemma}
\label{lem:ramsey}
    Fix a number $n$, a graph $G$ and a coloring
    $\chi\colon \hat G\to \str A^n$. Then there exists a graph $H$ and a coloring
    $\chi'\colon \hat H\to \str A^n$ such that
    for every coloring $\eta\colon\hat H\to \str A$, there is a polymorphism $f\colon \str A^n\to \str A$ and an embedding $\phi\colon G\to H$ such that 
    $$
    \chi'\circ\hat\phi = \chi
    \text{ and } 
    f\circ \chi = \eta\circ\hat\phi.
    $$
\end{lemma}
\begin{proof}
We define $\mathbf{G}:=(G,\chi,<)$ where $<$ is some linear order. Given a colored, ordered graph $\mathbf J = (J,\xi,<)$, we call a coloring $\eta\colon \hat J\to \str A$ \emph{good} if there is a polymorphism $f$ and a morphism $\phi\colon \mathbf{G}\to\mathbf{J}$ satisfying the condition above. Otherwise we call $\eta$ \emph{bad}.
    
    Let $\cat D\colon\cat C\op\to\set$ be the functor that sends $\mathbf{J}$ to the set of all bad colorings $\eta \colon\hat J\to \str A$.
    For each embedding $\iota\colon \mathbf{J'}\to\mathbf{J}$, let $\cat D_\iota$ be the map that restricts colorings $\eta$ along $\iota$:
    $$
    \cat D_\iota\colon \cat D(\mathbf{J})\to\cat D(\mathbf{J'}),\quad
    \eta\mapsto \eta\circ\hat\iota
    $$
    One may verify that, given a bad coloring $\eta$, the composition $\eta\circ\hat\iota$ is also bad.
    The claim will follow from Theorem~\ref{thm:konigramsey} once we show that this diagram has no solution, as this implies that $\cat D(\mathbf{H})$ is empty for some $\mathbf{H}$, meaning that all its colorings are good.
    Indeed, assume that $\{\eta_\mathbf{J}\}_{\mathbf{J}\in\cat C}$ is a solution. 
    For each tuple $a\in\str A^n$, let $\mathbf G_a$ be $K_2$ with the standard order and with its single edge $01$ colored by $a$. Then define a polymorphism $f$ as
    $$
    f(a)\coloneq \eta_{\mathbf{G}_a}(01).
    $$
    To show that this constitutes a polymorphism, consider any coloring $\chi\colon\hat K_m\to\str A^n$, and let $\mathbf{G}_\chi:= (K_m, \chi, <)$, where $<$ is the standard order on the $m$-element set.
    We show that  $f\circ \chi = \eta_{\mathbf{G}_\chi}$, which yields that $f\circ \chi$ is indeed a homomorphism $\hat K_m\to \str A$.
    Given an edge $vw$ of $K_m$ with
    $v < w$ there is a morphism in $\cat C$
    $$
    \iota\colon
    \mathbf{G}_{\chi(vw)}\to\mathbf{G}_\chi,\quad
    01\mapsto vw 
    $$
    Now compute the following:
    \begin{align*}
        f(\chi(vw))  = \eta_{\mathbf{G}_{\chi(vw)}} (01) 
        &= \cat D_\iota ( \eta_{\mathbf{G}_\chi})(01) \\&= \eta_{\mathbf{G}_\chi}(\hat\iota(01)) = 
        \eta_{\mathbf{G}_\chi}(vw)
    \end{align*}
    The first, third and fourth equalities hold by definition, the second because $\{\eta_\mathbf{J}\}_{\mathbf{J}\in\cat C}$ is a solution. But the same computation shows that $\eta_{\mathbf{G}}=f\circ\chi$, contradicting the assumption that $\eta_{\mathbf{G}}$ is bad.
\end{proof}

\mapFactory*

\begin{proof}
    Let $G_\bullet$ be a graph consisting of $|A^n|$ many disjoint edges and let $\chi_\bullet\colon\hat G_\bullet\to\str A^n$ be a surjective homomorphism. By Lemma~\ref{lem:ramsey} there is a graph $H$ and a coloring $\chi\colon\hat H\to\str A^n$, such that for all colorings $\eta\colon\hat H\to \str A$, there is a polymorphism $f$ and an embedding $\phi$, such that
    $$
    \chi\circ\hat\phi = \chi_\bullet
    \text{ and } 
    f\circ \chi_\bullet = \eta\circ\hat\phi.
    $$
    Pick a coloring $\eta\colon \hat H\to \str A$ such that $\eta(vw)\in U_{\chi(vw)}$ for all edges $vw\in\hat H$. Then for every $a\in\str A^n$ there is an edge $xy\in \hat G_\bullet$ with $a=\chi_\bullet(xy)=\chi(\hat\phi(xy))$, so
    $$
     f(a) = f(\chi_\bullet(xy)) = \eta(\hat\phi(xy))
     \in U_{\chi(\hat\phi(xy))} = U_a
    $$
\end{proof}

\lemRamsey*

\begin{proof}
Let $G_\bullet$ be the disjoint union of all graphs $G_i$ and let $\chi_\bullet\colon \hat G_\bullet\to\str A^n$ be the disjoint union of the colorings $\chi_i$. Let $J$ and $\xi\colon \hat J\to\str A^n$ be the graph and coloring obtained from Lemma~\ref{lem:ramsey}. We construct the desired graph $H$ as follows. Take the disjoint union of $G_\bullet$ and $J$. For every embedding $\iota\colon G_\bullet\to J$\ with $\xi\circ\hat\iota=\chi_\bullet$, introduce a copy of $G_\bullet \times P_2$ and identify:
\begin{itemize}
    \item every $v\in G_\bullet$ with $(v,0)\in G_\bullet\times P_2$ and
    \item every $(v,2)\in G_\bullet\times P_2$ with $\iota(v)\in J$
\end{itemize} 
Let $\chi\colon\hat H\to\str A^n$ be the following coloring:
$$
\chi(e) = 
\begin{cases}
    \chi_\bullet(vw) &\text{if } e = (vw,ij)\in\widehat{G_\bullet\times P_2}\\ 
    \xi(e) &\text{if } e\in \hat J
\end{cases}
$$
Observe that $\chi$ is indeed a homomorphism, as every clique in $H$ is fully contained in either $J$ or one of the $G_\bullet\times P_2$.

Let the embeddings $\phi_i\colon G_i\to H$ be the respective inclusions $G_i\subseteq G_\bullet\subseteq H$. The distance between different $G_i$ is at least four, as any path has to go through $J$. For every coloring $\eta\colon \hat H\to\str A$ consider its restrictions $\eta_J$  and $\eta_\bullet$ to $J\subseteq H$ and $G_\bullet\subseteq H$ respectively. The copy of $G_{\bullet}$ is the one from which this construction started; in Figure \ref{fig:G_S} it is the copy on the left.
By Lemma~\ref{lem:ramsey} there is a polymorphism $f\colon\str A^n\to \str A$ and an embedding $\iota\colon G_\bullet\to J$ such that $f\circ\chi_\bullet = \eta_J\circ \hat\iota$. Restricting $\eta$ to the copy of $G_\bullet\times P_2$ in $H$ corresponding to $\iota$, yields a 2-step cut-homotopy $\eta_\bullet\sim_2 f\circ\chi_\bullet$. Restricting this cut-homotopy to $G_i\times P_2\subseteq G_\bullet\times P_2$ yields cut-homotopies
$
\eta\circ \hat \phi_i\sim_2 f\circ\chi_i 
$.

\end{proof}
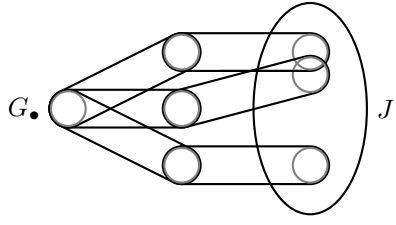
\begin{figure}
    \centering
\begin{tikzpicture}[line width=0.8pt]

\coordinate (A) at (0,0);
\coordinate (B) at (1.5,0.75);
\coordinate (C) at (1.5,0);
\coordinate (D) at (1.5,-0.75);
\coordinate (E) at (3.2,0.45);
\coordinate (F) at (3.2,0.75);
\coordinate (G) at (3.2,-0.75);
 
\draw (3.2,0) ellipse (0.75cm and 1.4cm);
 
\draw
  (-0.1118,0.2236)
  -- (1.3882,0.9736)
  arc[start angle=116.57, delta angle=-180, radius=0.25]
  -- (0.1118,-0.2236)
  arc[start angle=296.57, delta angle=-180, radius=0.25]
  -- cycle;
 
\draw
  (0,0.25)
  -- (1.5,0.25)
  arc[start angle=90.00, delta angle=-180, radius=0.25]
  -- (0,-0.25)
  arc[start angle=270.00, delta angle=-180, radius=0.25]
  -- cycle;
 
\draw
  (0.1118,0.2236)
  -- (1.6118,-0.5264)
  arc[start angle=63.43, delta angle=-180, radius=0.25]
  -- (-0.1118,-0.2236)
  arc[start angle=243.43, delta angle=-180, radius=0.25]
  -- cycle;
 
\draw
  (1.5,-0.5)
  -- (3.2,-0.5)
  arc[start angle=90.00, delta angle=-180, radius=0.25]
  -- (1.5,-1)
  arc[start angle=270.00, delta angle=-180, radius=0.25]
  -- cycle;
 
\draw
  (1.5,1)
  -- (3.2,1)
  arc[start angle=90.00, delta angle=-180, radius=0.25]
  -- (1.5,0.5)
  arc[start angle=270.00, delta angle=-180, radius=0.25]
  -- cycle;
\draw[gray] (F) circle (0.23cm);
 
\draw
  (1.4361,0.2417)
  -- (3.1361,0.6917)
  arc[start angle=104.83, delta angle=-180, radius=0.25]
  -- (1.5639,-0.2417)
  arc[start angle=284.83, delta angle=-180, radius=0.25]
  -- cycle;
\draw[gray] (E) circle (0.23cm);
 
\draw (2,0.5) -- (3.1,0.5);
 
\draw[gray] (D) circle (0.23cm);
\draw[gray] (A) circle (0.23cm);
\draw[gray] (B) circle (0.23cm);
\draw[gray] (G) circle (0.23cm);
\draw[gray] (C) circle (0.23cm);
 
\node[anchor=east]  at (-0.23, 0)    {$G_\bullet$};
\node[anchor=west]  at (3.95, 0)     {$J$};
 
\end{tikzpicture}
    \caption{The construction of $H$.}
    \label{fig:G_S}
\end{figure}

Recall that we use the notation $(vw, ij)$ for the edge $(v, i)(w, j)$ of a graph $G\times H$.
\begin{lemma}
\label{lem:homotopyFactory}
    Let $k$ and $n$ be numbers and $U_{a,ij} \subseteq A$ for all ${a}\in A^n$ and $ij \in \hat{P_k}$. Assume that for every graph $G$ and every coloring $\chi\colon \hat G \to \str A^n$, there is a coloring $\eta\colon \widehat{G\times P_k}\to \str A$ such that 
    $$
    \eta(vw,ij)\in U_{\chi(vw), ij}
    $$ 
    for all edges $(vw,ij)\in\widehat{G\times P_k}$.
    Then there exists a homomorphism 
    $h \colon \str A^n \times \hat{P_k} \to \str A$ 
    with $h({a}, ij)\in U_{{a}, ij}$ for all $a \in A^n$ and  $ij \in \hat{P_k}$.
\end{lemma} 

In particular, for each pair of edges $vw$ in $G$ and $ij$ in $P_k$ with $i\not=j$, the assumptions require that both $\eta(vw,ij)$ and $\eta(vw, ji)$ belong to $U_{\chi(vw),ij}$.

\begin{proof}
    Let $\cat D\colon \cat C\op\to\set$ be the functor that sends each linearly ordered $\str A^n$-colored graph $\mathbf{G}\coloneq (G,<,\chi)$ to the set of all homomorphisms $\eta\colon \widehat{G\times P_k} \to \str A$ that satisfy $\eta(vw,ij)\in U_{\chi(vw),ij}$ for all $vw \in \hat G$ and $ij\in\hat P_k$. 
    This set is finite and, by assumption, nonempty.
    For a morphism $\iota \colon \mathbf G \to \mathbf H$ define the map
    $$
    \cat D_\iota\colon \cat D(\mathbf{H})\to \cat D(\mathbf{G})
    $$
    to be the precomposition with the embedding $\widehat{(\iota\times \mathrm{id})}\colon \widehat{G\times P_k} \to \widehat{H \times P_k}$.
    \begin{center}
        \begin{tikzcd}
            \widehat{G\times P_k} \arrow[r, "\widehat{(\iota\times\mathrm{id})}"] \arrow[rd, "\cat D_\iota(\eta)"'] & \widehat{H \times P_k} \arrow[d, "\eta"]\\
            & \str A
        \end{tikzcd}
    \end{center}
    The category $\cat C$ is confluent and Ramsey, hence, by Theorem~\ref{thm:konigramsey}, there exists a solution to $\cat D$. Let $\{\eta_{\mathbf{G}}\}_{\mathbf G \in \cat C}$ be such a solution. 
    From this we construct the homomorphism $h\colon \str A^n \times \hat P_k \to \str A$ as follows.
    Given $a\in \str A^n, ij\in \hat P_k$ with $i\leq j$, we consider $\mathbf{G}_{a} \coloneq (K_2, 0<1, 01 \mapsto a)$ and define
    \begin{equation*}
        h(a,ij) \coloneq \eta_{\mathbf{G}_{a}}(01, ij).
    \end{equation*}
    Note that the left-hand side of the definition does not depend on the order of $i$ and $j$, but the right-hand side does, as $(01,ij)\neq (01,ji)$, by our naming convention for edges of $G \times P_k$. 
    Note that $\eta_{\mathbf{G}_{a}}(01, ij)$ belongs to $U_{a,ij}$ by definition, so it remains to check that $h$ is a homomorphism.
    To this end it suffices to check that for every $K_m$, every coloring $\chi\colon \hat K_m\to \str A^n$, and every $k$-cut $c\colon K_m\to P_k$, the mapping $h \circ (\chi,\hat c)$ is a homomorphism $\hat K_m \to \str A$.
    We define $\mathbf{G}_\chi \coloneq (K_m,\chi, <_c)$ where $<_c$ is any order that satisfies $v<_c w$ for all $v,w \in K_m$ with $c(v)<c(w)$.
    In other words, $<_c$ 
    is the order that makes $c$ monotone with respect to the standard order $\leq$ on the vertices of $P_k$.
    For all $v, w \in K_m$ with $v<_c w$, we have a morphism
    \begin{equation*}
    \iota\colon\mathbf{G}_{\chi(vw)} \to\mathbf{G}_\chi,\quad
    01\mapsto vw,
    \end{equation*}
    which we use to compute 
    \begin{align*}
    h(\chi(vw), \hat c(vw) ) &= 
    \eta_{\mathbf{G}_{\chi(vw)}}( 01, \hat{c}(vw))\\&=
    \cat D_\iota (\eta_{\mathbf{G}_{\chi}})( 01, \hat{c}(vw)) \\&= 
    \eta_{\mathbf{G}_\chi} \circ \widehat{(\iota\times\mathrm{id})}( 01, \hat{c}(vw)) \\&=
    \eta_{\mathbf{G}_{\chi}}( vw, \hat{c}(vw))\\&=
    \eta_{\mathbf{G}_\chi} \circ \widehat{(\mathrm{id},c)} (vw).
    \end{align*}
    Note that the first equality holds because $c(v) \leq c(w),$ the second one holds because $\{\eta_{\mathbf G}\}_{\mathbf G \in \cat C}$ constitutes a solution of $\cat D$ while the remaining equalities are simply expansions of definitions.
    This shows that $h\circ (\chi,\hat c)\colon\hat K_m \to \str A$ equals the map $\eta_{\mathbf{G}_\chi} \circ \widehat{(\mathrm{id},c)}$, 
    which is a homomorphism, since it is the composition of two homomorphisms.
\end{proof}

\propGadgetOne*

\begin{proof}
Fix two polymorphisms $f,g\colon \str A^n\to\str A$ that are not $k$-step cut-homotopic and set 
$$
U_{b,ij}\coloneq
\begin{cases}
    \{f(b)\} &\text{if } ij = 00\\
    \{g(b)\} &\text{if } ij = kk\\
    A &\text{otherwise,}
\end{cases}
$$
whenever $ij$ is an edge in $P_k$ and $b\in\str A^n$. By the contraposition of Lemma~\ref{lem:homotopyFactory}, there is a graph $G_{fg}$ and a coloring $\chi_{fg}\colon \hat G_{fg}\to\str A^n$ such that there is no coloring $\eta\colon \widehat{G_{fg}\times P_k}\to \str A$ with $\eta(vw,00) = f(\chi_{fg}(vw))$ and $\eta(vw,kk) = g(\chi_{fg}(vw))$. In other words, there is no $k$-step cut-homotopy between $f\circ\chi_{fg}$ and $g\circ \chi_{fg}$.

We then find the desired $G$ by taking the disjoint union of all $G_{fg}$ where $f,g$ are not $k$-step cut-homotopic and taking
$\chi\colon \hat G\to \str A^n$
to be the disjoint union of all $\chi_{fg}$. 
\end{proof}

\begin{remark}
    Although Theorem~\ref{thm:konigramsey} is a compactness lemma, its choice-free cousin, Proposition~3.1 in~\cite{Hadek_konig:ramsey}, suffices for all applications above. From its proof, one can construct the graphs we obtain in Propositions~\ref{prop:homotopy_witness} and~\ref{prop:homotopy_witness2} explicitly as iterated Ramsey witnesses. 
\end{remark}

\bibliographystyle{plain}
\bibliography{bib}
\end{document}